\documentclass{amsart}
\newtheorem{theorem}{Theorem}[section]

\newtheorem{proposition}[theorem]{Proposition}
\newtheorem{lemma}[theorem]{Lemma}
\newtheorem{problem}[theorem]{Open Problem}
\theoremstyle{remark}
\newtheorem{remark}[theorem]{Remark}

\usepackage{amsmath,amssymb,amsfonts,mathrsfs}
\usepackage{pgf,tikz}
\usepackage{graphicx}

\usepackage{hyperref}

\newcommand{\R}{\mathbb{R}}
\newcommand{\N}{\mathbb{N}}
\newcommand{\C}{\mathbb{C}}

\DeclareMathOperator*{\essinf}{{\rm ess}\,{\rm inf}}

\newcommand{\Beta}{\mathrm{B}}

\DeclareMathOperator{\divergence}{div}

\DeclareMathOperator{\argmin}{arg\,min}
\DeclareMathOperator{\Span}{span}

\newcommand{\Res}[2]{\mathcal{A}[{#1},{#2}]} 

\title[The Bilaplacian with Robin boundary conditions]{The Bilaplacian with Robin boundary conditions}
\author{Davide Buoso}
\author{James B.~Kennedy}

\address{Davide Buoso, Dipartimento di Scienze e Innovazione Tecnologica, Universit\`a degli Studi del Piemonte Orientale ``A.\ Avogadro'', viale Teresa Michel 11, 15121 Alessandria, Italy}
\email{davide.buoso@uniupo.it}

\address{James B.~Kennedy, Grupo de F\'isica Matem\'atica \textit{and} Departamento de Matem\'atica, Faculdade de Ci\^encias, Universidade de Lisboa, Campo Grande, Edif\'icio C6, 1749-016 Lisboa, Portugal}
\email{jbkennedy@fc.ul.pt}


\begin{document}

\begin{abstract}
We introduce Robin boundary conditions for biharmonic operators, which are a model for elastically supported plates and are closely related to the study of spaces of traces of Sobolev functions. We study the dependence of the operator, its eigenvalues, and eigenfunctions on the Robin parameters. We show in particular that when the parameters go to plus infinity the Robin problem converges to other biharmonic problems, and obtain estimates on the rate of divergence when the parameters go to minus infinity. We also analyse the dependence of the operator on smooth perturbations of the domain, computing the shape derivatives of the eigenvalues and giving a characterisation for critical domains under volume and perimeter constraints. We include a number of open problems arising in the context of our results.
\end{abstract}

\maketitle

\section{Introduction}
\label{sec:intro}

The classical Robin problem for the Laplacian on a bounded domain $\Omega\subset\mathbb R^d$
\begin{equation}
\label{robinlaplacian}
\begin{aligned}
-\Delta u &=f && \text{in\ }\Omega,\\
\frac{\partial u}{\partial \nu} &= - \gamma u && \text{on\ }\partial\Omega,
\end{aligned}
\end{equation}
where $\nu$ is the outer unit normal to $\partial\Omega$, plays the role of an interpolation between its Neumann and Dirichlet counterparts, where the former is given by the limiting case $\beta\equiv0$, and the latter is obtained in a suitable limiting sense when $\gamma\to+\infty$. Physically speaking, in two dimensions these problems can be derived as models for the displacement $u$ of a membrane of shape $\Omega$ subject to an external load $f$ whose boundary is elastically supported, free, or fixed, respectively, problems which admit a natural hierarchy. In this situation the Robin parameter $\gamma$ may be taken as a non-negative function defined on the boundary $\partial \Omega$ which is related with the elastic response to the displacement $u$ from the position at rest, but is commonly taken as a constant; mathematically, it makes sense to consider real or even complex parameters. Properties of the Robin problem, in particular as regards the dependence of the associated operator upon the parameter $\gamma$, have been intensively studied in the literature, and we refer the reader to the survey \cite{bucur17} as well as, e.g., \cite{bkl19,cch,dk10,emp14,filinovskiy14,filinovskiy15,hkr17,khalile18,los98,lp08,pankpop,pankpop16} and the references therein for further information.

If instead of membranes we wish to model plates, then the prototype equation now involves the Bilaplacian $\Delta^2$, and Neumann and Dirichlet boundary conditions still represent the cases of a free boundary and of a clamped boundary, respectively. However, the boundary conditions for the Bilaplacian are dramatically different from their Laplacian counterparts: indeed, while the Dirichlet problem reads
\begin{equation*}
\begin{aligned}
\Delta^2 u &=f && \text{in\ }\Omega,\\
\frac{\partial u}{\partial \nu} &= 0 && \text{on\ }\partial\Omega,\\
u &= 0 && \text{on\ }\partial\Omega,
\end{aligned}
\end{equation*}
the Neumann problem now involves rather complicated boundary operators (cf.\ \cite{buoso16,chasman11,chasman15,verchota})
\begin{equation}
\label{neumannintro}
\begin{aligned}
\Delta^2u &=f && \text{in\ }\Omega,\\
(1-\sigma)\frac{\partial^2u}{\partial\nu^2}+\sigma\Delta u &=0 && \text{on\ }\partial\Omega,\\
-\frac{\partial\Delta u}{\partial\nu}-(1-\sigma){\divergence}_{\partial\Omega}\frac{\partial\ }{\partial\nu}\nabla_{\partial\Omega} u &=0 && \text{on\ }\partial\Omega,
\end{aligned}
\end{equation}
where $\nabla_{\partial\Omega}$ denotes the tangential gradient to $\partial\Omega$ and ${\divergence}_{\partial\Omega}$ denotes the tangential divergence (see \cite[Section 2]{chasman11} for a derivation of \eqref{neumannintro}, see also \cite[Ch. 8]{delzol} for more details on these tangential operators). We remark that the fact that this problem is of fourth order always requires the specification of two boundary conditions to be well posed. This in turn gives rise to the possibility of different combinations of boundary conditions, in particular those of \emph{intermediate type}, such as the so-called Navier boundary conditions, which arise when the plate is hinged (but not clamped), see e.g., problems \eqref{navier-robin}--\eqref{dirichlet}. Moreover, we observe here the appearance of an additional coefficient $\sigma\in(-\frac{1}{d-1},1)$, called the \emph{Poisson ratio}, which is related to the stiffness properties of the material. While in the Dirichlet case this coefficient has no effect whatsoever on the problem, in the other cases it is an important parameter. We refer to \cite{ggs} for further discussion of the Poisson coefficient (see also \cite{prozkala}).

A further natural variant of the problem arises when the plate is under stress; in this case, the equation $\Delta^2u=f$ changes, resulting in a lower order perturbation:
\begin{equation}
\label{perturbed}
\Delta^2 u - \alpha \Delta u=f.
\end{equation}
Here $\alpha\in\mathbb R$, where $\alpha>0$ corresponds to a plate under tension, while for $\alpha<0$ it is under compression instead.
The dependence on the tension parameter $\alpha$ has been considered in various contexts, both for the analysis of the behaviour of the eigenvalues and focusing on limiting regimes, see \cite{abf19,buoso16,buchapro,bupro,chasman11,chasman15,kaw}. We note also that this perturbation arises naturally in the context and study of buckling phenomena for plates, where $\alpha$ becomes a so-called eigenvalue of the operator pencil
\begin{equation}
\label{buckling}
\Delta^2 u=-\Lambda \Delta u.
\end{equation}
Problem \eqref{buckling} has been widely studied in the literature, as it shows features that stand right in between the Laplacian and the Bilaplacian. While we are interested in the dependence upon the tension parameter $\alpha$, we shall not go into the details of the  properties of buckling eigenvalues, and we refer to \cite{ggs,henrot} for a more complete picture of the problem and for historical references.

A natural question at this point to ask is what form Robin boundary conditions for Bilaplacians should take, or more generally for perturbed Bilaplacians as in \eqref{perturbed}. Starting from \eqref{neumannintro} and performing an analysis similar to that of the Laplacian, as we shall see below in more detail, one is quickly led to the following problem:
\begin{equation}
\label{robin1strongintro}
\begin{aligned}
\Delta^2u-\alpha\Delta u &=f && \text{in\ }\Omega,\\
(1-\sigma)\frac{\partial^2u}{\partial\nu^2}+\sigma\Delta u &=-\beta\frac{\partial u}{\partial\nu} && \text{on\ }\partial\Omega,\\
\alpha\frac{\partial u}{\partial\nu}-\frac{\partial\Delta u}{\partial\nu}-(1-\sigma){\divergence}_{\partial\Omega}\frac{\partial\ }{\partial\nu}\nabla_{\partial\Omega} u &=-\gamma u && \text{on\ }\partial\Omega.
\end{aligned}
\end{equation}
Similarly to problem \eqref{robinlaplacian}, when $d=2$ the equation \eqref{robin1strongintro} models the behaviour of a three-dimensional thin plate of cross-section $\Omega$ subject to the load $f$, and the solution $u$ represents the displacement of the section of the plate with respect to its position at rest. The constants $\alpha,\beta,\gamma,\sigma$ are mechanical constants related to the response of the material with respect to mechanical stimulations; in particular, as mentioned, the Poisson ratio $\sigma$ measures the stiffness of the material, and $\alpha$ is the ratio of tension to flexural rigidity. The constants $\beta$ and $\gamma$, on the other hand, measure the elastic response of the boundary, in the transversal and normal direction, respectively. In this sense, in analogy with the case of the Laplacian \eqref{robinlaplacian}, $\beta$ and $\gamma$ will be called the Robin parameters of the problem.

To the best of our knowledge, problem \eqref{robin1strongintro} has never appeared in the mathematical literature yet, even though its derivation seems quite natural (note that a very recent preprint \cite{chaslang} studies isoperimetric properties in the important special case $\sigma=\beta=0$, $\alpha>0$). It is our objective in this paper formally to introduce the operator associated with this problem and begin an investigation into its properties, in particular as regards its eigenvalues. In doing so we will focus on a set of questions which seem to be of most interest either at an intrinsic level or most natural based on what is known for the Robin Laplacian, such as dependence upon the parameters, limits with respect to the parameters as these diverge to plus or minus infinity, and dependence on the domain. Since we are undertaking a first investigation of a necessarily limited scope, more questions and conjectures will arise than can be treated here. We will thus include a number of open problems throughout the paper which, we believe, merit a more detailed study in the future.

First and foremost, we introduce a bilinear form (see \eqref{qf}) as the natural fourth-order analogue of the form associated with problem \eqref{robinlaplacian}; in turn, this form indeed has problem \eqref{robin1strongintro} as its strong formulation. In particular, we show that problem \eqref{robin1strongintro} is indeed associated with a self-adjoint operator on $L^2(\Omega)$ with compact resolvent, for any possible choice of the parameters (Theorem \ref{lem:robin-form}). Moreover, as should be expected, the dependence of the operator upon the parameters is smooth, so that the eigenvalues can be organised into analytic branches (Theorem \ref{lem:continuity-monotonicity}).

Once such a smooth dependence is established, the next natural question concerns the limiting behaviour, that is, what happens for $\beta\to\pm\infty$ and $\gamma\to\pm\infty$. In the case of the Robin Laplacian \eqref{robinlaplacian}, where there is only one parameter $\gamma$, the two extremities correspond to two very different situations: on the one hand, when $\gamma \to+\infty$ the operator converges to the Dirichlet Laplacian, and we have information on the rate of convergence of the eigenvalues and eigenfunctions: for example, at least on sufficiently smooth domains, the difference between the $n$-th Robin and the $n$-th Dirichlet Laplacian eigenvalue $\lambda_n$ is bounded by $C_\Omega \lambda_n^2 \gamma^{-1}$, where the constant $C_\Omega$ is independent of $n$ (see e.g., \cite{filinovskiy14,filinovskiy15}, or \cite[Section~4.4.1]{bucur17} for a summary). On the other hand, while the operator is still bounded from below for each fixed $\gamma<0$, as $\gamma \to -\infty$ there is now an infinite family of eigenvalues diverging to $-\infty$ like $-C_\Omega \gamma^2$, where the constant $C_\Omega \geq 1$ depends on smoothness properties of the boundary, with $C_\Omega = 1$ in the regular case (see, e.g., \cite{bkl19,dk10,emp14,hkr17,khalile18,los98,lp08,pankpop,pankpop16}, as well as \cite[Section~4.4.2]{bucur17}). However, once we exclude the divergent analytic branches, the remaining ones are strictly positive and accumulate at the Dirichlet Laplacian spectrum (see \cite[Theorem~1.5]{bkl19}).

In the case of the Bilaplacian, the two parameters are independent, so we can fix one and study the other. In particular, we see that if we fix $\beta$ and let $\gamma\to+\infty$, we obtain an intermediate problem of Robin type with parameter $\beta$ that is strongly linked to the Navier problem, which we shall call a \emph{Navier--Robin problem}, see \eqref{navier-robin}. If we reverse the order, namely we fix $\gamma$ and let $\beta\to+\infty$, we obtain another intermediate problem (see \eqref{kuttler-sigillito}), to which we will attach the names of Kuttler and Sigillito as it was introduced in \cite{kuttler68}. In both cases, if we then let the other parameter also tend to plus infinity, we recover the Dirichlet problem, showing that, despite the complications arising owing to the presence of multiple parameters, the Robin Laplacian and the Robin Bilaplacian share a common limiting behaviour in the ``plus infinity'' direction (see Theorems~\ref{thm:positive-convergence-1} and~\ref{thm:positive-convergence-2}).

When diverging towards minus infinity, however, the situation becomes much more involved. Firstly, we observe that for each $k \in \N$ the $k$-th eigenvalue of problem \eqref{robin1strongintro} diverges to minus infinity (Theorem~\ref{thm:steklovbuckling}). The most interesting result is perhaps the rate of divergence of the eigenvalues: for any $k\in\mathbb N$ we will show that $\lambda_k\sim |\beta|^4$ as $\beta\to-\infty$, and $\lambda_k \sim |\gamma|^{4/3}$ as $\gamma\to-\infty$, see Theorems~\ref{thm:divergence-basic},~\ref{thm:num-range-eig-div}, and~\ref{thm:divergence-1}. This asymptotic behaviour will be obtained by combining lower bounds obtained for the numerical range of the operator, generalising the arguments of \cite[Section~6]{bkl19}, with upper bounds given by suitable choices of test functions inspired by \cite{dk10}. 
Here we will not cover the intermediate Navier--Robin \eqref{navier-robin} and the Kuttler--Sigillito \eqref{kuttler-sigillito} problems, which require different techniques and will be left to a later work. However, they should display the same behaviour, as hinted by Theorem~\ref{thm:num-range-eig-div} (cf.\ Remark~\ref{rem:intermediate-regime}).
We will also leave for future work a closer analysis of the convergence behaviour in order to derive a more precise asymptotic expansion for the eigenvalues. The natural method for such an analysis would be to use Dirichlet-Neumann bracketing to isolate the behaviour at the boundary, as was done in \cite{pankpop} for the Robin Laplacian \eqref{robinlaplacian}; developing such a technique in the considerably more involved case of the Bilaplacian will require special treatment.

It is worth noting that the tension parameter $\alpha$ has a completely different nature. This parameter is clearly related to the behaviour in the interior of the plate rather than the boundary, so that no link is expected with the Robin coefficients, and this is a point of additional interest in the limiting situations for $\alpha$ as well. The regime $\alpha\to+\infty$ has been partially considered in the literature as a singular perturbation of the Laplacian (meaning that the order of the equation changes in the limit, cf.\ \cite{frank}), and here we examine this situation in more depth. If the Robin parameters are kept bounded, then unsurprisingly in the limit they simply disappear, making $\alpha$ the predominant parameter (see Theorem \ref{conv-alpha}). If, on the other hand, we posit a different relationship, in particular if we suppose that the Robin parameters also diverge linearly in $\alpha$, then we can recover Laplacians of Robin and Neumann type in the limit. We refer in particular to Theorem \ref{thm:strange-robin}, where in the case of less smooth domains a ``new'' kind of Laplacian emerges whose form domain involves a decoupling between the interior and the boundary of the domain; this fits neatly into the theory of $j$-elliptic forms developed in \cite{ate11,ate12}.

On the other hand, the direction $\alpha\to-\infty$ seems not to have been investigated so much. Recently, the behaviour of the eigenvalues for the Dirichlet problem when $\alpha\to-\infty$ was studied in \cite{abf19}. There, the analytic branches were completely identified when the domain is a ball, and it was shown for any domain and any $k \in \N$ that $\lambda_k +|\alpha|^2/4\in o(|\alpha|^2)$, meaning in particular that the asymptotic behaviour does not depend in any manner on the domain. It was also shown in \cite{abf19} that the same asymptotics holds for the classical Navier problem, that is, the Navier problem with $\sigma=1$. In the general Robin case we see that indeed the behaviour of the eigenvalues is $\mathcal O(|\alpha|^2)$, and we obtain $-1/4$ as an upper bound for the coefficient (Theorems~\ref{thm:divergence-basic},~\ref{thm:num-range-eig-div}(a), and~\ref{thm:divergence-1}(a)). In fact, the parabola $-|\alpha|^2/4$ acts as a separator between the Dirichlet problem and the Neumann (and more generally Robin) one. In \cite{abf19}, the precise value of the coefficient was obtained via an inclusion property that, unfortunately, is not available for boundary conditions other than Dirichlet ones. We nevertheless expect the same coefficient to apply in general; see Remark~\ref{rem:big-neg-alpha}.

We also investigate the dependence of the eigenvalues on smooth perturbations of the domain. While the proof of the shape continuity of the eigenvalues can be easily done as in \cite{cohi}, the question of analyticity is more subtle. In fact, while simple eigenvalues are analytic (see e.g., \cite{henry}), multiple ones show the occurrence of bifurcation phenomena, and can be split into analytic branches, but only when the parametrisation is given by one real parameter (cf.\ \cite[Theorem 1]{rellich}). General families of perturbations, though, cannot be parametrised in this way; so to overcome this problem we pass to the use of elementary symmetric functions of the eigenvalues, which bypass altogether such splitting phenomena, and can indeed be shown to be analytic (cf.\ \cite{buoso16,bula,lala2004,lala2007}). On the other hand, if the perturbation is not smooth, the spectrum may exhibit singular behaviour (see e.g., \cite{arfela,arla,fela}).

The analyticity of the symmetric functions of the eigenvalues allows us to consider their shape derivatives, that in turn lead to Hadamard-type formulae for the Robin problem \eqref{robin1strongintro} (Theorem \ref{duesettesys}). We then address the question of shape optimisation for eigenvalues. While for the Laplacian we have now quite a decent understanding of the subject (see \cite{bradep,henrot}), the theory is very much underdeveloped for the Bilaplacian, with comparatively few results available in the literature \cite{ash, bufega,buga11,buchapro,bupro,chasman11,chasman15,kuttler72,nadir}. However, combining the shape derivatives with the Lagrange Multiplier Theorem we obtain a nice characterisation of critical domains with respect to a volume constraint or a perimeter constraint (Theorem \ref{moltiplicatorisys}), and the fact that the operator is rotation invariant implies that balls are critical domains for any eigenvalue of \eqref{robin1strongintro} (Theorem \ref{lepallesys}; cf.\ \cite{bu15,bula}). This is in accordance with similar results already proven for other types of biharmonic problem, see \cite{buoso16} and the references therein.

Another similarity that the Robin Bilaplacian shares with the Robin Laplacian is the link with the Steklov problem, equivalently, the duality between the eigenvalues of the Robin Laplacian and the eigenvalues of the Dirichlet-to-Neumann operator, where the coefficient in the boundary condition becomes the spectral parameter. In the case of the Laplacian, this connection has been exploited in numerous contexts (as in, for example, the recent works \cite{bkl19,fl1,fl2}, among many others). This connection is rather straightforward for the Laplacian in the sense that there is a unique Steklov problem, and only one boundary parameter. The Bilaplacian, on the other hand, encompasses three completely different Steklov-type problems, two of which were introduced in \cite{kuttler68} and the third in \cite{bupro}. Indeed, for each of the three different Robin problems we consider, the ``Full Robin'' problem where $\beta,\gamma \in \R$, the Navier--Robin problem with $\gamma=+\infty$ and the Kuttler--Sigillito problem with $\beta=+\infty$, we obtain a different Steklov problem. But more than that, in the Robin case $\beta,\gamma \in \R$ there are of course two potential spectral parameters in the boundary conditions rather than just one. All this suggests that this biharmonic Robin-Steklov link should be far more intricate than for the Laplacian, and correspondingly more interesting. Here we do not exploit this link in detail, we just observe that it emerges naturally e.g., when studying divergence properties of the eigenvalues of Robin Bilaplacians, cf.\ Theorem~\ref{thm:steklovbuckling}. We also note in passing that this connection has a direct application to studying traces of $H^2$-functions, where these problems become the focus of the whole analysis. We will not cover this matter though, as the aim of the paper is different, and we refer the interested reader to the recent works \cite{laproz, lapro}, where in fact special cases of problem \eqref{robin1strongintro} do appear in the context of Steklov problems; more precisely, there the authors take $f=0$, and one of the Robin parameters is now treated as the eigenvalue, considered as dependent on the other parameter.

The paper is organised as follows. In Section \ref{sec:setting} we introduce the Robin problem \eqref{robin1strongintro} as well as the variants when $\beta=+\infty$ or $\gamma=+\infty$, together with the associated operators on $L^2(\Omega)$, and establish a number of elementary properties, in particular discreteness of the spectrum and smooth dependence on the parameters. In Section \ref{sec:large-positive} we prove the convergence to other biharmonic problems as the Robin parameters go to plus infinity, while Section \ref{sec:conv-lapl} is devoted to the case of convergence to the Laplacian as the tension parameter $\alpha$ goes to plus infinity. In Section \ref{sec:num-range} we establish the existence of divergent negative eigenvalues, and their rate of divergence, for large negative values of the parameters. Finally, in Section \ref{sec:shape-derivatives} we compute Hadamard-type formulae for the eigenvalues and study critical domains under volume and perimeter constraints.


\section{Robin-type boundary conditions for the Bilaplacian}
\label{sec:setting}

Let $\Omega\subset \mathbb R^d$ be a bounded domain with Lipschitz boundary. Let also $\alpha,\beta,\gamma\in \mathbb{R}$ and $\sigma\in(-\frac{1}{d-1},1)$. We introduce the following bilinear quadratic form
\begin{multline}
\label{qf}
\mathcal Q_{\Omega,\sigma,\alpha,\beta,\gamma}(u, v )=\int_{\Omega} \left((1-\sigma)D^2u:D^2 v +\sigma\Delta u\Delta v \right)\,dx
+\alpha\int_{\Omega} \nabla u\cdot \nabla v  \,dx\\
+\beta \int_{\partial\Omega}\frac{\partial u}{\partial\nu}\frac{\partial  v }{\partial\nu} \,d\mathcal{H}^{d-1}(x)
+\gamma \int_{\partial\Omega} u  v  \,d\mathcal{H}^{d-1}(x),
\end{multline}
defined for all $u \in H^2 (\Omega)$. Here $\nu$ denotes the outer normal vector to $\partial\Omega$ and $\mathcal{H}^{d-1}$ the $(d-1)$-dimensional surface (Hausdorff) measure on it, while $D^2u:D^2 v $ denotes the Frobenius product
$$
D^2u:D^2 v =\sum_{i,j=1}^d\frac{\partial^2u}{\partial x_i\partial x_j}\frac{\partial^2 v }{\partial x_i\partial x_j}.
$$
As is the case for the Neumann Bilaplacian, the restriction on $\sigma$ will be necessary to ensure coercivity of the form \eqref{qf} (see \cite{ggs,prozkala} and cf. also the proof of Theorem~\ref{lem:robin-form}, in particular \eqref{coerciveeq}, where one can see where the optimal constant $-1/(d-1)$ comes from).

We will call the following problem the (weak form of the) Robin problem for the Bilaplacian (or \emph{Full Robin problem}, FR):
\begin{equation}
\label{robin1weak}
\mathcal Q_{\Omega,\sigma,\alpha,\beta,\gamma}(u, v )=\int_\Omega f v  \,dx,
\end{equation}
for any function $ v \in H^2(\Omega)$, with $f\in L^2(\Omega)$. 
Problem \eqref{robin1weak} has the following strong formulation (cf.\ \eqref{robin1strongintro})
\begin{equation}
\label{robin1strong}
\begin{cases}
\Delta^2u-\alpha\Delta u=f, & \text{in\ }\Omega,\\
\Beta(u) = -\beta\frac{\partial u}{\partial\nu}, & \text{on\ }\partial\Omega,\\
\Gamma(u)=-\gamma u, & \text{on\ }\partial\Omega,
\end{cases}
\end{equation}
where we have used the short-hand notation
\begin{equation*}
\begin{aligned}
	\Beta(u) &:= (1-\sigma)\frac{\partial^2u}{\partial\nu^2}+\sigma\Delta u,\\
	\Gamma(u) &:= \alpha\frac{\partial u}{\partial\nu}-\frac{\partial\Delta u}{\partial\nu}-(1-\sigma){\divergence}_{\partial\Omega}\frac{\partial\ }{\partial\nu}\nabla_{\partial\Omega} u.
\end{aligned}
\end{equation*}

We also consider the same equation $\Delta^2 u -\alpha \Delta u = f$ equipped with any one of the three following sets of boundary conditions, corresponding formally to the cases $\gamma = +\infty$, $\beta = +\infty$, and $\beta = \gamma = +\infty$, respectively: we shall refer to these as the \emph{Navier--Robin} problem (or NR for short)
\begin{equation}
\label{navier-robin}
\begin{cases}
\Delta^2u-\alpha\Delta u=f, & \text{in\ }\Omega,\\
\Beta(u) = -\beta\frac{\partial u}{\partial\nu}, & \text{on\ }\partial\Omega,\\
u=0, & \text{on\ }\partial\Omega,
\end{cases}
\end{equation}
the \emph{Kuttler--Sigillito} (KS) problem
\begin{equation}
\label{kuttler-sigillito}
\begin{cases}
\Delta^2u-\alpha\Delta u=f, & \text{in\ }\Omega,\\
\frac{\partial u}{\partial\nu}=0, & \text{on\ }\partial\Omega,\\
\Gamma(u)=-\gamma u, & \text{on\ }\partial\Omega,
\end{cases}
\end{equation}
and the (classical) Dirichlet problem
\begin{equation}
\label{dirichlet}
\begin{cases}
\Delta^2u-\alpha\Delta u=f, & \text{in\ }\Omega,\\
u=\frac{\partial u}{\partial\nu}=0, & \text{on\ }\partial\Omega.\\
\end{cases}
\end{equation}

We are using the name {\it Navier--Robin problem} for \eqref{navier-robin} due to the obvious analogy with the classical Navier problem
$$
\begin{cases}
\Delta^2u=f, & \text{in\ }\Omega,\\
u=\Delta u=0, & \text{on\ }\partial\Omega,
\end{cases}
$$
which in turn can be obtained from \eqref{navier-robin} by setting $\alpha=\beta=0$ and $\sigma=1$ (see also \cite{abf19,buoso16}). On the other hand, the NR problem is also related to what is known as the classical Steklov Bilaplacian
\begin{equation}
\label{stek}
\begin{cases}
\Delta^2u-\alpha\Delta u=0, & \text{in\ }\Omega,\\
\Beta(u) = \lambda\frac{\partial u}{\partial\nu}, & \text{on\ }\partial\Omega,\\
u=0, & \text{on\ }\partial\Omega,
\end{cases}
\end{equation}
which has received much attention in recent years (cf.\ \cite{bufega,buga11,buoso16}). The eigenvalue problem \eqref{stek} was actually introduced in \cite{kuttler68}, where the authors proved a series of inequalities between eigenvalues of different second-order and fourth-order operators. We remark that the Rayleigh quotient associated with \eqref{stek} was first introduced in \cite{fichera}, but in a different context and with different purposes.

We shall call problem \eqref{kuttler-sigillito} the {\it Kuttler--Sigillito problem} since it was also introduced in \cite{kuttler68}, cf.\ \cite{lapro}, albeit as a Steklov-type problem and in the case $\alpha=0$, $\sigma=1$.

We also remark that the weak formulations of the Full Robin (FR) problem \eqref{robin1strong}, of the NR problem \eqref{navier-robin}, of the KS problem \eqref{kuttler-sigillito}, and of the Dirichlet problem \eqref{dirichlet} all involve the same form \eqref{qf} on the form domains $H^2(\Omega)$, $H^2(\Omega) \cap H^1_0 (\Omega)$, $\{u \in H^2 (\Omega) : \frac{\partial u}{\partial \nu}=0\}$, and $H^2_0 (\Omega)$, respectively, under the convention that $0 \cdot (+\infty) = 0$. As mentioned, we will always consider the latter three problems as (formally) corresponding to the cases $\gamma=+\infty$, $\beta=+\infty$, and $\beta=\gamma=+\infty$, respectively, a notational convention which will be justified by the convergence results of Section~\ref{sec:large-positive}.

Let us start by summarising a few basic properties of the form $\mathcal{Q}_{\Omega,\sigma,\alpha,\beta,\gamma}$.

\begin{theorem}
\label{lem:robin-form}
Fix $\alpha \in \R$, $\beta,\gamma \in (-\infty,+\infty]$ and $\sigma \in (-\frac{1}{d-1},1)$. The form $\mathcal{Q}_{\Omega,\sigma,\alpha,\beta,\gamma} : H^2 (\Omega) \times H^2 (\Omega) \to \R$ given by \eqref{qf} is bilinear, symmetric, continuous and $L^2$-elliptic (that is, coercive after adding a fixed multiple of the square of the $L^2(\Omega)$-norm).
\end{theorem}

\begin{proof}
Symmetry and bilinearity are clear. Continuity follows directly from the trace inequality, which allows the $L^2 (\partial\Omega)$-norm of both the trace and the normal derivative of a function $u \in H^2 (\Omega)$ to be controlled by $\|u\|_{H^2(\Omega)}$ (see \cite[Section 20]{beilni2}). So we only have to prove the $L^2$-ellipticity, that is, we seek constants $\omega \in \R$ and $c>0$ such that
\begin{equation}
\label{form-l2-ellipticity}
	\mathcal{Q}_{\Omega,\sigma,\alpha,\beta,\gamma} (u,u) + \omega \|u\|_2^2 \geq c\|u\|_{H^2 (\Omega)}^2
\end{equation}
for all $u$ in the form domain of $\mathcal{Q}_{\Omega,\sigma,\alpha,\beta,\gamma}$, where we use the notation $\|u\|_2 := \|u\|_{L^2(\Omega)}$. 
To obtain \eqref{form-l2-ellipticity}, we start by observing that since
\begin{equation}
\label{coerciveeq}
	\|D^2u\|_2^2\ge \frac 1 d \|\Delta u\|_2^2,
\end{equation}
the expression
\begin{equation}
\label{h2equivnorm}
	\|u\|_{2}^2 + (1-\sigma)\|D^2u\|^2_{2} + \sigma \|\Delta u\|^2_{2}, \qquad u \in H^2(\Omega),
\end{equation}
defines an equivalent norm on $H^2 (\Omega)$ for any $\sigma \in (-\frac{1}{d-1},1)$. 

We now consider the case $\beta,\gamma \in \R$; we fix $\alpha,\beta,\gamma \in \R$. Now the embedding $H^2(\Omega) \hookrightarrow H^1(\Omega)$ and the trace mapping $H^2(\Omega) \to L^2(\partial\Omega) \times L^2(\partial\Omega)$, $u \mapsto (u|_{\partial\Omega}, \frac{\partial u}{\partial\nu})$ are both compact, since $\Omega$ is assumed bounded and Lipschitz. Hence, by a standard Ehrling's lemma-type argument, we may find, for any $\varepsilon>0$, a constant $C=C(\Omega,\varepsilon,\alpha,\beta,\gamma,\sigma)$ such that
\begin{equation}
\label{traceinequality-collection}
	|\alpha| \|\nabla u\|_{2}^2 + |\beta| \left\|\frac{\partial u}{\partial\nu}\right\|_{2,\partial\Omega}^2 +
	|\gamma| \|u\|_{2,\partial\Omega}^2 \leq \varepsilon \left[(1-\sigma)\|D^2u\|^2_{2} + \sigma 	
	\|\Delta u\|^2_{2}\right] + C\|u\|_{2}^2
\end{equation}
for all $u \in H^2(\Omega)$, where $\|v\|_{2,\partial\Omega} := \|v\|_{L^2(\partial\Omega)}$. Choosing $\varepsilon \in (0,1)$ and using that \eqref{h2equivnorm} defines an equivalent norm on $H^2(\Omega)$ yields \eqref{form-l2-ellipticity} for some $\omega \geq 0$ sufficiently large.

An entirely analogous argument shows that if $\gamma = +\infty$, then \eqref{traceinequality-collection} still holds for all $u \in H^2 (\Omega) \cap H^1_0 (\Omega)$ (where the product $\gamma u^2$ on $\partial\Omega$ is interpreted as being identically zero, consistent with the form), and hence \eqref{form-l2-ellipticity}; the cases $\beta = +\infty$ and $\beta = \gamma = +\infty$ are, likewise, completely analogous.
\end{proof}

It follows (see \cite[Section~VI.2]{kato76}) that the associated operator on $L^2(\Omega)$, which we will call the \emph{Robin Bilaplacian} in the case $\beta,\gamma \in\mathbb R$, is self-adjoint and bounded from below, and it has compact resolvent due to the compactness of the embedding $H^2(\Omega) \hookrightarrow L^2(\Omega)$; therefore its spectrum consists of a divergent sequence of eigenvalues
\begin{equation*}
	\lambda_1 \leq \lambda_2 \leq \lambda_3 \leq \ldots \to +\infty,
\end{equation*}
where we repeat each according to its finite multiplicity. Unlike in the case of the Laplacian, there is no reason to expect the first eigenvalue to be simple or to have an associated positive eigenfunction (see \cite{buopar,ggs} and the references therein).  Moreover, as we shall see, if one or more parameters are negative, a finite number of eigenvalues may be negative (cf.\ Theorem~\ref{thm:steklovbuckling}). In any case, the eigenfunctions $u_k$ corresponding to $\lambda_k$ may be chosen in such a way that they form an orthonormal basis of $L^2 (\Omega)$. The eigenvalues may be characterised by the usual Courant--Fischer min-max principle
\begin{equation}
\label{minmax}
\lambda_k(\Omega,\sigma,\alpha,\beta,\gamma)
=\min_{\substack{V\subset H^2(\Omega) \\ \dim V=k}}\  \max_{0\neq u \in V}
\frac{\mathcal Q_{\Omega,\sigma,\alpha,\beta,\gamma}(u,u)}
{\int_{\Omega} u^2 \,dx}.
\end{equation}
In what follows, in order to simplify the notation, we will drop any or all of the arguments of $\lambda_k(\Omega,\sigma,\alpha,\beta,\gamma)$ and indices of $\mathcal Q_{\Omega,\sigma,\alpha,\beta,\gamma}$ whenever they are clear from the context.


In particular, the eigenvalue problem corresponding to \eqref{robin1weak} has the following weak formulation
\begin{equation}
\label{robinweak}
\mathcal Q_{\Omega,\sigma,\alpha,\beta,\gamma}(u, v )
=\lambda\int_\Omega u v  \,dx,
\end{equation}
for any function $ v \in H^2(\Omega)$, and its strong formulation reads
\begin{equation}
\label{robinstrong}
\begin{cases}
\Delta^2u-\alpha\Delta u=\lambda u, & \text{in\ }\Omega,\\
(1-\sigma)\frac{\partial^2u}{\partial\nu^2}+\sigma\Delta u=-\beta\frac{\partial u}{\partial\nu}, & \text{on\ }\partial\Omega,\\
\alpha\frac{\partial u}{\partial\nu}-\frac{\partial\Delta u}{\partial\nu}-(1-\sigma){\divergence}_{\partial\Omega}\frac{\partial\ }{\partial\nu}\nabla_{\partial\Omega} u=-\gamma u, & \text{on\ }\partial\Omega,
\end{cases}
\end{equation}
with completely analogous statements for the limit problems corresponding to \eqref{navier-robin}, \eqref{kuttler-sigillito} and \eqref{dirichlet}. Similarly to problem \eqref{robin1strong}, in the two-dimensional case $d=2$ problem \eqref{robinstrong} models the vibrations of a three-dimensional thin plate of cross-section $\Omega$, where the eigenvalues $\lambda_k$ are the eigenfrequencies and the associated eigenfunctions $u_k$ are the eigenmodes (we refer back to the Introduction for a physical interpretation of the parameters).

\begin{theorem}
\label{lem:continuity-monotonicity}
Fix $\sigma \in (-\frac{1}{d-1},1)$ and $k\in\mathbb N$. For each of the variables $\alpha,\beta,\gamma \in \R$, $\lambda_k$ is a piecewise analytic (in particular continuous) and monotonically increasing function of that variable, if the other variables are held constant. 

In addition, both the resolvent and the operator are continuous with respect to the triple $(\alpha,\beta,\gamma) \in \R^3$.
\end{theorem}

\begin{proof}
Monotonicity is an immediate consequence of the monotonicity of the form $Q_{\Omega,\sigma,\alpha,\beta,\gamma} (u,u)$ in $\alpha$, $\beta$ and $\gamma$ for each fixed $u \in H^2 (\Omega)$, together with the characterisation \eqref{minmax}.

For the piecewise analyticity of the eigencurves we use Kato's theory of analytic perturbation of operators \cite{kato76}. We give the argument for variable $\alpha$ and fixed $\beta,\gamma$; the other cases are completely analogous. First note that for each fixed $u \in H^2(\Omega)$, the mapping $\alpha \mapsto Q_{\Omega,\sigma,\alpha,\beta,\gamma} (u,u)$ is linear, and hence analytic, in $\alpha \in \R$. Together with Theorem~\ref{lem:robin-form}, this means that for $\alpha \in \R$ the family of Robin operators is holomorphic of type (B), and self-adjoint holomorphic, in the sense of Kato; see Section~VII.4 and in particular Theorem~VII.4.2 and Remark~VII.4.7 of \cite{kato76}. The claims about the eigencurves now follow from \cite[Theorem~VII.1.8 and Section~VII.3.1]{kato76}.

The continuity of the resolvent and of the operator follows directly from the fact that the family of Robin operators is holomorphic of type (B), and self-adjoint holomorphic, with respect to any of the variables.
\end{proof}

\begin{remark}
We observe that the $\lambda_k$ satisfy a local joint analyticity property, in the following sense: for any point $(\alpha_0,\beta_0,\gamma_0) \in \R^3$ where $\lambda_k$ is separately analytic in a neighbourhood of each variable, then it is analytic with respect to the triple $(\alpha,\beta,\gamma) \in \R^3$ in a neighbourhood of $(\alpha_0,\beta_0,\gamma_0) \in \R^3$. This is an immediate consequence of Hartogs' Theorem on separate holomorphy (see, e.g., \cite{krantz}) applied to the natural complex extension of the parameters and, correspondingly, of the operator.
\end{remark}

\begin{remark}
A natural question at this point is whether problems~\eqref{robin1strong}--\eqref{dirichlet} admit semiclassical asymptotic expansions as $k\to+\infty$ and under which conditions. For any fixed admissible value of the parameters, setting $\omega_d$ the Lebesgue measure of the unit ball, the classical Weyl limit
\begin{equation}
\label{eq:weyl-limit}
\lambda_k(\Omega,\sigma,\alpha,\beta,\gamma)=(2\pi)^4\left(\frac{k}{\omega_d|\Omega|}\right)^{\frac 4 d}+o\left(k^{\frac 4 d}\right)
\end{equation}
can be easily inferred from the analogous result for the Dirichlet problem~\eqref{dirichlet} and the Neumann problem (i.e., problem~\eqref{robin1strong} with $\beta=\gamma=0$) computed e.g., in~\cite{buprostu}. In particular, as is well known, the asymptotic expansion~\eqref{eq:weyl-limit} only depends on the principal part of the operator, so that different types of boundary conditions do not affect such an expansion (see e.g., \cite{lapflec,safvas}). Further terms in the expansion~\eqref{eq:weyl-limit} can be computed, and are well known to depend on the principal part of the operators appearing in the boundary conditions, see \cite[Chapter 1]{safvas}, meaning that in fact they do not depend on either $\beta$ or $\gamma$ (as long as they are finite). However, for the second term in the asymptotic expansion to make sense one needs much stronger regularity hypotheses on the domain $\Omega$ that are usually given in terms of billiard trajectories. Notice that the cases when $\beta$ and/or $\gamma$ vanish (if not equal to plus infinity) have been treated in \cite[Section 3]{buprostu}, and such expansions can be directly applied to problems~\eqref{robin1strong}--\eqref{dirichlet}, when the regularity assumptions on $\Omega$ are satisfied (see \cite{buprostu,safvas} for more details on assumptions and computations).
\end{remark}


\section{Convergence for large positive values of the Robin parameters}
\label{sec:large-positive}

In this section we prove convergence of the eigenvalues $\lambda_k$ to the eigenvalues of the Bilaplacian with Navier--Robin \eqref{navier-robin}, Kuttler--Sigillito \eqref{kuttler-sigillito}, and Dirichlet \eqref{dirichlet} boundary conditions as $\gamma \to +\infty$, $\beta \to +\infty$, and $\beta,\gamma \to +\infty$, respectively, if $\sigma$ and $\alpha$ are fixed. We will also explicitly cover the special cases in which one of the parameters is already fixed at plus infinity, that is, the case of NR to Dirichlet, where $\gamma=+\infty$ and $\beta\to+\infty$, and the case of KS to Dirichlet, where $\beta=+\infty$ and $\gamma\to+\infty$. Throughout, it will always be understood that $\frac 1 \infty=0$.

In order to simplify notation we will suppress the parameters $\sigma$ and $\alpha$; the respective forms are then $\mathcal{Q}_{\beta,+\infty}$, defined on $H^2(\Omega) \cap H^1_0 (\Omega)$, $\mathcal{Q}_{+\infty,\gamma}$, on $\{u \in H^2(\Omega): \frac{\partial u}{\partial \nu} = 0\}$, and $\mathcal{Q}_{+\infty,+\infty}$, on $H^2_0 (\Omega)$. 
Likewise, we will write $\lambda_k (\beta,+\infty)$, $\lambda_k (+\infty,\gamma)$ and $\lambda_k(+\infty,+\infty)$, respectively, for the corresponding eigenvalues. The corresponding eigenfunctions, chosen to form an orthonormal basis of $L^2(\Omega)$, will be denoted by $u_{k,\beta,\gamma}$, where $\beta$ and $\gamma$ may take on the value $+\infty$. Also, assuming without loss of generality that $-1$ is not an eigenvalue (if it is, we may simply shift the operator by a multiple of the identity, or equivalently by adding an $L^2$-term to the form), for fixed $\sigma$ and $\alpha$ we will write
\begin{displaymath}
	\Res{\beta}{\gamma}
\end{displaymath}
for the resolvent operator at $-1$, that is, for $f \in L^2(\Omega)$, $u = \Res{\beta}{\gamma}f$ is the unique solution of $\Delta^2u - \alpha\Delta u + u =f$ subject to the boundary conditions $\Beta (u) = \beta\frac{\partial u}{\partial\nu}$, $\Gamma (u) = \gamma u$; its eigenvalues will be denoted by
\begin{equation}
\label{eq:lambda-mu}
	\mu_k (\beta,\gamma) = \frac{1}{\lambda_k (\beta,\gamma)+1},
\end{equation}
where, again, $\beta$ and $\gamma$ may take on the value $+\infty$. Our main convergence result is as follows. The cases of NR to Dirichlet as $\beta\to +\infty$ and KS to Dirichlet as $\gamma \to +\infty$ are contained in parts (a) and (b), respectively.

\begin{theorem}
\label{thm:positive-convergence-1}
Suppose $\Omega \subset \R^d$ is a Lipschitz domain and $\sigma \in (-\frac{1}{d-1},1)$ and $\alpha \in \R$ are fixed. Then $\lambda_k(\beta,\gamma)$ is jointly monotonically increasing for $(\beta,\gamma) \in (-\infty,+\infty] \times (-\infty,+\infty]$, and
\begin{enumerate}
\item[(a)] for each fixed $\beta \in (-\infty,+\infty]$, for every $k\geq 1$ we have $\lambda_k (\beta,\gamma) \to \lambda_k (\beta,+\infty)$ as $\gamma \to +\infty$; moreover, there exists an eigenfunction $u_{k,\beta,+\infty}$ for $\lambda_k (\beta, +\infty)$ such that up to a subsequence, $u_{k,\beta,\gamma} \rightharpoonup u_{k,\beta,+\infty}$ weakly in $H^2(\Omega)$;
\item[(b)] for each fixed $\gamma \in (-\infty,+\infty]$, for every $k\geq 1$ we have $\lambda_k (\beta,\gamma) \to \lambda_k (+\infty,\gamma)$ as $\beta \to +\infty$; moreover, there exists an eigenfunction $u_{k,+\infty,\gamma}$ for $\lambda_k (+\infty,\gamma)$ such that up to a subsequence, $u_{k,\beta,\gamma} \rightharpoonup u_{k,+\infty,\gamma}$ weakly in $H^2(\Omega)$;
\item[(c)] if $\beta,\gamma \to +\infty$ jointly, then for every $k\geq 1$ we have $\lambda_k (\beta,\gamma) \to \lambda_k (+\infty,+\infty)$; moreover, there exists an eigenfunction $u_{k,+\infty,+\infty}$ for $\lambda_k (+\infty,+\infty)$ such that up to a subsequence, $u_{k,\beta,\gamma} \rightharpoonup u_{k,+\infty,+\infty}$ weakly in $H^2(\Omega)$.
\end{enumerate}
\end{theorem}

\begin{theorem}
\label{thm:positive-convergence-2}
Suppose in addition to the assumptions of Theorem~\ref{thm:positive-convergence-1} that $\partial\Omega$ is $C^4$. Then in cases (a) and (b) of Theorem~\ref{thm:positive-convergence-1} we also have convergence of the corresponding resolvents $\Res{\beta}{\gamma}$ in the operator norm, with respective estimates
\begin{displaymath}
\begin{aligned}
	\|\Res{\beta}{\gamma} - \Res{\beta}{+\infty}\| &\leq \frac{C_1(\Omega,\sigma,\alpha,\beta)}{\sqrt{\gamma}} \qquad &&\text{for all sufficiently large $\gamma$, in case (a),}\\ 
	\|\Res{\beta}{\gamma} - \Res{+\infty}{\gamma}\| &\leq \frac{C_2(\Omega,\sigma,\alpha,\gamma)}{\sqrt{\beta}}  \qquad &&\text{for all sufficiently large $\beta$, in case (b),}
\end{aligned}
\end{displaymath}
for constants $C_1(\Omega,\sigma,\alpha,\beta),C_2(\Omega,\sigma,\alpha,\gamma)>0$ depending only on the indicated parameters. In particular, there exist $\gamma_0 > 0$ in case (a) and $\beta_0 > 0$ in case (b) such that the eigenvalues satisfy the respective bounds
\begin{equation}
\label{eq:positive-ev-bound-1}
	0 \leq \lambda_k(\beta,+\infty) - \lambda_k(\beta,\gamma) \leq \frac{C_1(\Omega,\sigma,\alpha,\beta)\lambda_k(\beta,+\infty)^2}{\sqrt{\gamma}}
\end{equation}
for all $\gamma\geq\gamma_0$ and all $k \in \N$ in case (a), and
\begin{equation}
\label{eq:positive-ev-bound-2}
	0 \leq \lambda_k(+\infty,\gamma) - \lambda_k(\beta,\gamma) \leq \frac{C_2(\Omega,\sigma,\alpha,\gamma) \lambda_k(+\infty,\gamma)^2}{\sqrt{\beta}}
\end{equation}
for all $\beta\geq\beta_0$ and all $k \in \N$ in case (b).
\end{theorem}

The monotonicity of the eigenvalues with respect to the parameters $\beta$ and $\gamma$ contained in Theorem~\ref{thm:positive-convergence-1} is an extension of Theorem~\ref{lem:continuity-monotonicity}, where only the finite case was considered.
In particular, if either $\beta$ or $\gamma$ equals $+\infty$, then the statement remains true for the respective other parameter; while the form domains are always nested in the right way. 

The proof of Theorem~\ref{thm:positive-convergence-1} is based on a general argument about convergent forms, which also applies equally to the convergence of the eigenvalues of the Robin Laplacian to those of the Dirichlet Laplacian, and which is somehow reminiscent of an abstract version of Mosco convergence; for convenience of reference we will formulate this in abstract terms. We observe that Theorem~\ref{thm:positive-convergence-1} may be viewed as complementary to \cite[Theorems 3.7 and 3.12]{lapro}, which are based on similar ideas.

On the other hand, the proof of Theorem~\ref{thm:positive-convergence-2} relies on a completely different method inspired by \cite{filinovskiy14}. Regarding case (c) of Theorem~\ref{thm:positive-convergence-1}, we conjecture that an inequality of the type
$$
\|\Res{\beta}{\gamma} - \Res{+\infty}{+\infty}\| \leq C_3(\Omega,\sigma,\alpha)\left(\frac 1{\sqrt{\beta}}+\frac{1}{\sqrt{\gamma}}\right)
$$
holds, analogous to the inequalities in Theorem~\ref{thm:positive-convergence-2}. In principle, this may be obtained in exactly the same way as (a) and (b); however, it requires regularity estimates for elliptic partial differential equations of higher order where the dependence of the constants on the various coefficients can be controlled explicitly in the right way. Such results do not seem to be currently available in the literature (see \cite[Theorem 2.20]{ggs} for the statements without explicit control on the constants), and it would be too large an undertaking to derive a refined version of that theorem here. Note also that the constants $C_1$, $C_2$ appearing in the estimates depend on various trace and embedding estimates (cf.\ \eqref{eq:domain-estimate-1} and \eqref{eq:h4-control}), and thus it would be difficult to control them explicitly.

It should be possible to refine and extend Theorem~\ref{thm:positive-convergence-2} in several other different directions, as well. Since our principal goal is to give a broad outline of what to expect, we will restrict ourselves here to the simpler, more regular case, and leave refinements as open problems:

\begin{problem}
Determine the optimal form of the bounds \eqref{eq:positive-ev-bound-1} and \eqref{eq:positive-ev-bound-2}, including the optimal powers of $\beta$ and $\gamma$. Obtain an explicit estimate on the constants $C_1$ and $C_2$.
\end{problem}

\begin{problem}
Prove that Theorem~\ref{thm:positive-convergence-2} still holds when $\Omega$ has corners, or more generally is merely Lipschitz. Note that this problem appears to be open also for the Robin Laplacian (the papers \cite{filinovskiy14,filinovskiy15,filinovskiy15b} all deal with the smooth case). We remark that the technique used to deal with the smooth case relies on regularity estimates for the solution; it is not clear whether such a strategy could be successful in general.
\end{problem}

\begin{problem}
In the smooth case, obtain a precise asymptotic expansion for $\lambda_k(\beta,\gamma)$ in terms of $\lambda_k (\beta,+\infty)$ and powers of $\gamma$ as $\gamma \to +\infty$, where $\beta$ is fixed. Do the same when $\beta \to +\infty$, for fixed $\gamma$. (Cf., e.g., \cite[Theorem~1]{filinovskiy15b} for the Robin Laplacian.)
\end{problem}

\begin{problem}
Study the case when $\beta, \gamma \to +\infty$ simultaneously.
\end{problem}

\begin{lemma}
\label{lem:mosco-form}
Suppose $H,V,V_1,V_2,\ldots,V_\infty$ are Hilbert spaces satisfying the chain of continuous embeddings
\begin{equation}
\label{eq:embedding-chain}
	V_\infty \hookrightarrow \ldots \hookrightarrow V_2 \hookrightarrow V_1 \hookrightarrow V \hookrightarrow H,
\end{equation}
where the last embedding is compact. Suppose also that the symmetric sesquilinear form $\mathcal{Q}_n : V_n \times V_n \to \C$ is continuous and $H$-elliptic, $n=1,2,\ldots,\infty$, and that
\begin{enumerate}
\item $\mathcal{Q}_n (v) \leq \mathcal{Q}_{n+1} (v)$ for all $v \in V_{n+1}$ and all $n\in \N$, as well as $\mathcal{Q}_n (v) \leq \mathcal{Q}_\infty (v)$ for all $v \in V_\infty$ and $n \in \N$; and
\item whenever $v_n \in V_n$ with $\|v_n\|_H=1$ and the sequence $(\mathcal{Q}_n(v_n))_{n\in\N}$ is bounded from above, the sequence $(v_n)_{n\in\N}$ in $V$ admits a subsequence with a weak limit in $V_\infty$, and in this case $\mathcal{Q}_\infty (v) \leq \liminf_{n\to\infty} \mathcal{Q}_n (v_n)$.
\end{enumerate}
Denote by $\lambda_k (\mathcal{Q}_n)$ the $k$-th eigenvalue of the form $\mathcal{Q}_n$, $n=1,\ldots,\infty$, counted with multiplicities. Then:f
\begin{enumerate}
\item[(a)] $\lambda_k (\mathcal{Q}_n) \to \lambda_k (\mathcal{Q}_\infty)$ from below, for all $k=1,2,\ldots$;
\item[(b)] denoting by $u_{k,n}$ any eigenvector associated with $\lambda_k (\mathcal{Q}_n)$, normalised so that $\|u_{k,n}\|_H=1$, there exists an eigenvector $u_{k,\infty}$ for $\lambda_k (\mathcal{Q}_\infty)$ with $\|u_{k,\infty}\|_H = 1$ such that, up to a subsequence, $u_{k,n} \rightharpoonup u_{k,\infty}$ weakly in $V$ as $n \to \infty$.
\end{enumerate}
\end{lemma}

Here we will write $\mathcal{Q}_n(v)$ as shorthand for $\mathcal{Q}_n(v,v)$, and we will denote the inner product on $H$ by $(\cdot,\cdot)_H$. We recall that $\lambda$ is an eigenvalue for the form $\mathcal{Q}_n$ if there exists $u\in V_n$ such that $\mathcal{Q}_n(u,v)=\lambda (u,v)_H$ for all $v\in V_n$.

\begin{proof}
Observe first that the eigenvalues $\lambda_k(\mathcal{Q}_n)$ are given by the usual min-max formula; hence it follows directly from the inclusions \eqref{eq:embedding-chain} and (1) that, for each $k\geq 1$, $(\lambda_k(\mathcal{Q}_n))_{n \in \N}$ is a monotonically increasing sequence in $n$, bounded from above by $\lambda_k(\mathcal{Q}_\infty)$.

We will prove the claims by induction on $k$. For $k=1$, take a sequence $(u_{1,n})_{n\in\N}$ of eigenvectors, each normalised such that $\|u_{1,n}\|_H = 1$. Since $\mathcal{Q}_n(u_{1,n}) = \lambda_1 (\mathcal{Q}_n)$ is bounded in $n$, by (2), up to a subsequence, $(u_{1,n})_{n\in\N}$ has a weak limit $u^\ast \in V_\infty$; by compactness of the embedding $V \hookrightarrow H$, we also have $\|u^\ast\|_H = 1$. Moreover, $\mathcal{Q}_\infty (u^\ast) \leq \liminf_{n \to \infty} \mathcal{Q}_n (u_{1,n})$. Recalling that the $u_{1,n}$ are eigenvectors and that the sequence $(\lambda_1(\mathcal{Q}_n))_{n\in\N}$ is monotonically increasing, we see this means that
\begin{displaymath}
	\lambda_1 (\mathcal{Q}_\infty) \leq \mathcal{Q}_\infty (u^\ast) \leq \lim_{n \to \infty} \lambda_1(\mathcal{Q}_n) \leq \lambda_1 (\mathcal{Q}_\infty);
\end{displaymath} 
hence there is equality. Equality also implies that $u^\ast$ is an eigenvector associated with $\lambda_1 (\mathcal{Q}_\infty)$.

Now fix $k\geq 2$ suppose the statement is true for $j=1,\ldots,k-1$. We suppose that, up to a subsequence in $n$, the normalised eigenvectors $u_{j,n}$ of $\lambda_j(\mathcal{Q}_n)$ converge weakly in $V$ to the eigenvector $u_{j,\infty}$ of $\lambda_j(\mathcal{Q}_\infty)$ as $n\to\infty$, for all $j=1,\ldots,k-1$. Now we may repeat the argument used for $k=1$: up to a subsequence the sequence $(u_{k,n})_{n\in\N}$ has a weak limit $u_k^\ast \in V_\infty$ since $\mathcal{Q}_n(u_{k,n})=\lambda_k(\mathcal{Q}_n)$ is likewise bounded, and for this limit we have $\|u_k^\ast\|_H = 1$ and
\begin{displaymath}
	\mathcal{Q}_\infty (u_k^\ast) \leq \lim_{n \to \infty} \lambda_k(\mathcal{Q}_n) \leq \lambda_k (\mathcal{Q}_\infty).
\end{displaymath}
All the claims of the lemma will now follow from the min-max characterisation of $\lambda_k(\mathcal{Q}_\infty)$ if we can show that $(u_k^\ast,u_{j,\infty})_H=0$ for all $j=1,\ldots,k-1$. But since $(u_{j,n},u_{k,n})_H=0$ for all $j\neq k$ and all $n\in\N$, we have, by strong convergence and the Cauchy--Schwarz inequality in $H$,
\begin{equation*}
	|(u_{j,\infty},u_{k}^\ast)_H| \leftarrow |(u_{j,n},u_{k}^\ast)_H| = |(u_{j,n},u_{k,n})_H-(u_{j,n},u_{k}^\ast)_H| 
	      \leq \|u_{j,n}\|_H \|u_{k,n}-u_{k}^\ast\|_H \to 0
\end{equation*}
as $n \to \infty$. It follows that indeed $(u_{j,\infty},u_k^\ast)_H=0$ as required. This completes the proof.
\end{proof}

\begin{proof}[Proof of Theorem~\ref{thm:positive-convergence-1}]
In each case the statement can be proved through a direct application of Lemma~\ref{lem:mosco-form}. We will only give the proof of (a) in all details, as in the other two cases the arguments are essentially identical. Clearly, it suffices to prove the discrete version of (a), that is, we suppose $\beta \in (-\infty,+\infty]$ is fixed (in fact, for ease of exposition we will take $\beta \in \R$) and take a sequence $\gamma_n \to +\infty$. We then choose $H=L^2(\Omega)$, $V=V_1 = V_2 = \ldots = H^2(\Omega)$, $V_\infty = H^2(\Omega) \cap H^1_0 (\Omega)$, as well as
\begin{displaymath}
	\mathcal{Q}_n = \mathcal{Q}_{\beta,\gamma_n}, \qquad \mathcal{Q}_\infty = \mathcal{Q}_{\beta,+\infty}.
\end{displaymath}
Now it follows from our choices of $V_n$ and $\mathcal{Q}_n$ that, for any $v \in H^2(\Omega)$,
\begin{displaymath}
	\mathcal{Q}_{n+1} (v) - \mathcal{Q}_{n} (v) = \int_{\partial\Omega} (\gamma_{n+1} - \gamma_n)|v|^2\,dx \geq 0,
\end{displaymath}
while for any $v \in H^2 (\Omega) \cap H^1_0 (\Omega)$,
\begin{displaymath}
	\mathcal{Q}_n (v) = \int_\Omega (1-\sigma)|D^2 v|^2 + \sigma |\Delta v|^2 + \alpha |\nabla v|^2\,dx + \beta\int_{\partial\Omega}\left|\frac{\partial v}{\partial\nu}\right|^2
	\,d\mathcal{H}^{d-1}(x) = \mathcal{Q}_\infty (v);
\end{displaymath}
this means that condition (1) of Lemma~\ref{lem:mosco-form} is satisfied. For (2), we  let $v_n \in H^2 (\Omega)$ be any functions for which $\mathcal{Q}_n(v_n)$ forms a bounded sequence. We show then that, up to a subsequence, there exists some $v \in H^2(\Omega) \cap H^1_0 (\Omega)$ such that $v_n \rightharpoonup v$ weakly in $H^2(\Omega)$ as $n \to \infty$. Firstly note that the $L^2$-ellipticity of the forms $\mathcal{Q}_{\beta,0}$, together with the fact that $\mathcal{Q}_{n}(v) \geq \mathcal{Q}_{\beta,0}(v)$ for all (sufficiently large) $n \in \N$ and for all $v\in H^2(\Omega)$, means that $v_n$ also forms a bounded sequence in $H^2(\Omega)$ and hence up to a subsequence admits a weak limit $v \in H^2(\Omega)$. We claim that in fact $v \in H^2(\Omega) \cap H^1_0 (\Omega)$. To see this, observe that the boundedness of the sequence
\begin{displaymath}
	\gamma_n \int_{\partial\Omega} |v_n|^2\,dx
\end{displaymath}
implies that $v_n|_{\partial\Omega} \to 0$ in $L^2(\partial\Omega)$. But the compactness of the trace operator from $H^2(\Omega)$ to $L^2(\partial\Omega)$ and the weak convergence of $v_n$ to $v$ in $H^2(\Omega)$ means that $v_n \to v$ strongly in $L^2(\partial\Omega)$, and hence $v=0$ in $L^2(\partial\Omega)$, that is, $v \in H^1_0 (\Omega)$, as claimed.

It remains to show condition (2) of Lemma~\ref{lem:mosco-form}, that $\mathcal{Q}_\infty (v) \leq \liminf_{n\to\infty} \mathcal{Q}_n (v_n)$. But it follows from the weak convergence of the $v_n$ to $v$ in $H^2(\Omega)$ that
\begin{displaymath}
\begin{aligned}
	\int_\Omega (1-\sigma)|D^2v|^2 + \sigma |\Delta v|^2\,dx &\leq \liminf_{n\to\infty} \int_\Omega (1-\sigma)|D^2v_n|^2 + \sigma |\Delta v_n|^2\,dx,\\
	\alpha\int_\Omega |\nabla v|^2\,dx &= \lim_{n\to\infty} \alpha \int_\Omega|\nabla v_n|^2\,dx,\\
	\beta\int_{\partial\Omega} \left|\frac{\partial v}{\partial\nu}\right|^2\,dx &= \lim_{n\to\infty} \beta\int_{\partial\Omega} \left|\frac{\partial v_n}{\partial\nu}\right|^2\,dx,\\
	0 &\leq \liminf_{n\to\infty} \gamma_n \int_{\partial\Omega} |v_n|^2\,dx,
\end{aligned}
\end{displaymath}
where the first relation is due to the fact that \eqref{h2equivnorm} defines an equivalent norm on $H^2(\Omega)$, plus the strong convergence of the $L^2$-norms $\|v_n\|_2 \to \|v\|_2$; the second follows from the compactness of the embedding $H^2(\Omega) \hookrightarrow H^1(\Omega)$; the third from the compactness of the first-order trace mapping $H^2(\Omega) \to L^2(\partial\Omega)$, $u \mapsto \frac{\partial u}{\partial\nu}$; and the fourth is trivial since $\gamma_n \geq 0$. Summing these four relations completes the proof that (2) holds. Applying Lemma~\ref{lem:mosco-form} now immediately yields part (a) of the theorem.

For cases (b) and (c), we take $V_n$ and $\mathcal{Q}_n$ as before; but now $V_\infty = \{v \in H^2(\Omega): \frac{\partial v}{\partial\nu} = 0\}$ and $\mathcal{Q}_\infty = \mathcal{Q}_{+\infty,\gamma}$ in case (b), and $V_\infty = H^2_0 (\Omega)$ and $\mathcal{Q}_\infty = \mathcal{Q}_{+\infty,+\infty}$ in case (c). The rest of the argument is analogous.
\end{proof}

\begin{proof}[Proof of Theorem~\ref{thm:positive-convergence-2}]
As above, given $\sigma$, $\alpha$ and $\beta$, after a shift if necessary we may suppose without loss of generality that $\alpha\geq 0$, $\beta \in [0,+\infty]$ and $\gamma>0$ and, as above, that $-1$ is not an eigenvalue for any $\gamma \in (0, +\infty]$. For $h \in L^2 (\Omega)$ we set $u := \Res{\beta}{\gamma}h$ and $\tilde u := \Res{\beta}{+\infty}h$ as well as $w:=u-\tilde u = (\Res{\beta}{\gamma} - \Res{\beta}{+\infty})h$; then $u,\tilde u,w \in H^4 (\Omega)$ since $\partial\Omega\in C^4$ (see \cite[Theorem 2.20]{ggs}). Then $w$ satisfies the equation $\Delta^2 w - \alpha\Delta w + w = 0$ in $L^2(\Omega)$, and, upon multiplying by $w$, integrating over $\Omega$, integrating by parts, and using the boundary conditions that $u$ and $\tilde u$ satisfy, we obtain
\begin{multline*}
	-\int_\Omega w^2\,dx = \int_\Omega w(\Delta^2 w - \alpha\Delta w)\,dx
	= \int_\Omega (1-\sigma)|D^2w|^2+\sigma|\Delta w|^2 + \alpha |\nabla w|^2\,dx \\
	+ \int_{\partial\Omega} \beta \left|\frac{\partial w}{\partial\nu}\right|^2 
		+ w(-\Gamma(w))\,d\mathcal{H}^{d-1}(x).
\end{multline*}
Since by definition $w$ satisfies the boundary condition $w = -\frac{1}{\gamma} (\Gamma(\tilde u)+\Gamma(w))$ in $L^2(\partial\Omega)$, we obtain
\begin{multline*}
	\int_\Omega (1-\sigma)|D^2w|^2+\sigma|\Delta w|^2 + \alpha |\nabla w|^2 + w^2\,dx + \int_{\partial\Omega} \beta \left|\frac{\partial w}{\partial\nu}\right|^2 
	+ \frac{1}{\gamma} \Gamma(w)^2\,d\mathcal{H}^{d-1}(x) \\
= -\frac{1}{\gamma}\int_{\partial\Omega} \Gamma(w)\Gamma(\tilde u)\,d\mathcal{H}^{d-1}(x) \geq 0.
\end{multline*}
Since $\alpha,\beta\geq 0$, $\gamma > 0$ by assumption, this in turn implies
\begin{displaymath}
\begin{aligned}
	\int_\Omega w^2\,dx + \frac{1}{\gamma}\int_{\partial\Omega} \Gamma(w)^2\,d\mathcal{H}^{d-1}(x) 
	&\leq \frac{1}{\gamma}\left|\int_{\partial\Omega} \Gamma(w)\Gamma(\tilde u)\,d\mathcal{H}^{d-1}(x)\right|\\
	&\leq \frac{1}{2\gamma}\|\Gamma(w)\|_{2,\partial\Omega}^2 + \frac{1}{2\gamma} \|\Gamma(\tilde u)\|_{2,\partial\Omega}^2,
\end{aligned}
\end{displaymath}
leading to
\begin{displaymath}
	\|w\|_2 \leq \frac{1}{\sqrt{2\gamma}}\|\Gamma(\tilde u)\|_{2,\partial\Omega}.
\end{displaymath}
Using the trace estimates
\begin{equation}
\label{eq:domain-estimate-1}
	\left\|\frac{\partial \tilde u}{\partial\nu}\right\|_{2,\partial\Omega},\,\left\|\frac{\partial}{\partial\nu}\Delta \tilde u\right\|_{2,\partial\Omega},\,
	\left\|\divergence_{\partial\Omega} \frac{\partial}{\partial\nu} \nabla_{\partial\Omega}\tilde u\right\|_{2,\partial\Omega} \leq C(\Omega) \|\tilde u\|_{H^4(\Omega)}
\end{equation}
as well as the elliptic regularity estimate (cf.\ \cite[Theorem 2.20]{ggs})
\begin{equation}
\label{eq:h4-control}
	\|\tilde u\|_{H^4(\Omega)} \leq \tilde{C}(\Omega,\sigma,\alpha,\beta) \|h\|_2,
\end{equation}
and recalling the definition of $w$, we thus conclude that
\begin{displaymath}
	\|(\Res{\beta}{\gamma} - \Res{\beta}{+\infty}) h\|_2 \leq \frac{1}{\sqrt{2\gamma}}\|\Gamma(\tilde u)\|_{2,\partial\Omega} \leq \frac{C(\Omega)}{\sqrt{\gamma}}\|\tilde u\|_{H^4(\Omega)}
	\leq \frac{C_1(\Omega,\sigma,\alpha,\beta)}{\sqrt{\gamma}} \|h\|_2,
\end{displaymath}
for all $h \in L^2(\Omega)$ and all $\gamma > 0$. This means exactly that
\begin{equation*}
	\|\Res{\beta}{\gamma} - \Res{\beta}{+\infty}\| \leq \frac{C_1(\Omega,\sigma,\alpha,\beta)}{\sqrt{\gamma}}.
\end{equation*}
Adapting the above argument in the obvious way, it easily follows that also
\begin{equation*}
	\|\Res{\beta}{\gamma} - \Res{+\infty}{\gamma}\| \leq \frac{C_2(\Omega,\sigma,\alpha,\gamma)}{\sqrt{\beta}}
\end{equation*}
for some constant $C_2(\Omega,\sigma,\alpha,\gamma)>0$, for any fixed $\gamma \in (0,+\infty]$.

The argument to obtain the eigenvalue bounds is a simple one, following exactly the same lines as in \cite{filinovskiy14}. We briefly sketch the arugment in case (a); case (b) is,  again, completely analogous. It follows from the resolvent estimate that the corresponding eigenvalues $\mu_k(\beta,\gamma)$, as defined in \eqref{eq:lambda-mu}, satisfy the same bound: for sufficiently large $\gamma>0$ we have
\begin{displaymath}
	|\mu_k(\beta,\gamma) - \mu_k(\beta,+\infty)| \leq \frac{C_1(\Omega,\sigma,\alpha,\beta)}{\sqrt{\gamma}}.
\end{displaymath}
Recalling the definition of $\mu_k$, this rearranges to
\begin{displaymath}
	|\lambda_k(\beta,+\infty) - \lambda_k(\beta,\gamma)| \leq \frac{C_1(\Omega,\sigma,\alpha,\beta)}{\sqrt{\gamma}}(\lambda_k(\beta,+\infty)+1)(\lambda_k(\beta,\gamma)+1).
\end{displaymath}
The monotonicity statement $\lambda_k (\beta,\gamma) \leq \lambda_k(\beta,+\infty)$ now yields the conclusion.
\end{proof}


\section{Convergence for large positive values of the tension parameter}
\label{sec:conv-lapl}

In this section we focus on the behaviour of the various Robin problems when the tension parameter $\alpha$ diverges to plus infinity. As already mentioned, this parameter plays a fundamentally different role, and has a completely different effect on the family of operators, than the Robin parameters $\beta$ and $\gamma$; this will also be seen in the techniques and the results in the respective situations. The strategy we follow here is to write the operator associated with the Robin problem as a perturbation of the type
$$
L+\varepsilon G,
$$
where up to a renormalisation $L$ represents the Laplacian term (or more precisely the associated bilinear form) and $G$ the Bilaplacian and boundary terms, and then take the limit as $\varepsilon\to 0+$.  We will distinguish between two cases: where $\beta$ and $\gamma$ remain bounded as $\alpha \to +\infty$, and model cases where they diverge like $\alpha$.

In order to set the notation, we recall that the Dirichlet Laplacian eigenvalue problem reads
\begin{equation*}
\begin{cases}
-\Delta v= \mu^D v, & \text{\ in }\Omega,\\
v=0, & \text{\ on }\partial\Omega,
\end{cases}
\end{equation*}
while the Neumann one reads
\begin{equation*}
\begin{cases}
-\Delta v= \mu^N v, & \text{\ in }\Omega,\\
\frac{\partial v}{\partial\nu}=0, & \text{\ on }\partial\Omega,
\end{cases}
\end{equation*}
and the Robin one is (cf.\ \eqref{robinlaplacian})
\begin{equation}
\label{eq:robin-laplacian}
\begin{cases}
	-\Delta v = \mu^R v, &\text{\ in } \Omega,\\
	\frac{\partial v}{\partial\nu} + \tilde\gamma v = 0 \qquad &\text{\ on } \partial\Omega,
\end{cases}
\end{equation}
for some given $\tilde\gamma \in \R$. We number the Dirichlet, Neumann and Robin Laplacian eigenvalues as increasing sequences (counted with multiplicities), and denote them by $\mu_k^D$, $\mu_k^N$, and $\mu_k^R = \mu_k^R(\tilde\gamma)$, respectively.

\medskip

\textbf{The case of bounded $\beta$ and $\gamma$.}

\begin{theorem}
\label{conv-alpha}
Let $\beta,\gamma\in\mathbb R$ be fixed (or, more generally, bounded), let $\sigma\in(-\frac{1}{d-1},1)$, and let $\Omega \subset \R^d$ be a bounded domain with Lipschitz boundary. Then the following statements hold.
\begin{enumerate}
\item Let $\lambda_k(\alpha)$ be the $k$-th eigenvalue of either the NR problem \eqref{navier-robin} or the Dirichlet problem \eqref{dirichlet}. Then
\begin{equation*}
\lim_{\alpha\to+\infty}\frac{\lambda_k(\alpha)}{\alpha}=\mu_k^D.
\end{equation*}
Moreover, the eigenprojection associated with $\lambda_k(\alpha)$ converges in $L^2(\Omega)$ to the eigenprojection associated with $\mu_k^D$.

\item Let $\lambda_k(\alpha)$ be the $k$-th eigenvalue of either the Full Robin problem \eqref{robin1strong} or the KS problem \eqref{kuttler-sigillito}. Then
\begin{equation*}
\lim_{\alpha\to+\infty}\frac{\lambda_k(\alpha)}{\alpha}=\mu_k^N.
\end{equation*}
Moreover, the eigenprojection associated with $\lambda_k(\alpha)$ converges in $L^2(\Omega)$ to the eigenprojection associated with $\mu_k^N$.
\end{enumerate}
\end{theorem}

\begin{proof}

Since all the problems we consider share the same quadratic form $\mathcal Q = \mathcal Q_{\Omega,\sigma,\alpha,\beta,\gamma}$ (defined in \eqref{qf}), it is sufficient to handle this form generically and then specialise with respect to the form domain depending on the particular problem we want to consider. We define
\begin{equation*}
L(u, v )=\int_{\Omega} \nabla u\cdot \nabla v  \,dx,
\end{equation*}
and
\begin{equation*}
G_{\sigma,\beta,\gamma}(u, v )=\int_{\Omega} \left((1-\sigma)D^2u:D^2 v +\sigma\Delta u\Delta v \right)\,dx
+\beta \int_{\partial\Omega}\frac{\partial u}{\partial\nu}\frac{\partial  v }{\partial\nu} \,d\mathcal{H}^{d-1}(x)
+\gamma \int_{\partial\Omega} u  v  \,d\mathcal{H}^{d-1}(x),
\end{equation*}
where for the meantime both are defined on the natural form domain of $G=G_{\sigma,\beta,\gamma}$, namely $H^2_0(\Omega)$ for the Dirichlet problem, $H^2(\Omega)\cap H^1_0(\Omega)$ for the NR problem, $\{u\in H^2(\Omega): \frac{\partial u}{\partial\nu}=0\}$ for the KR problem, and $H^2(\Omega)$ for the Full Robin problem. 
Clearly, we have that $\mathcal Q=G+\alpha L$. Moreover, both $L$ and $G$ are quadratic forms densely defined in $L^2(\Omega)$ and, depending on the form domain, can be associated with self-adjoint operators with compact resolvent: $G$ can be associated with the Bilaplacian (in the case $\alpha=0$), while the Friedrichs extension of the operator associated with $L$ (i.e., the operator associated with $L$ on the closure of the form domain, $H^1_0 (\Omega)$ in the Dirichlet and NR cases, $H^1(\Omega)$ in the other cases) is the Laplacian with Dirichlet conditions in the first two cases and Neumann ones in the latter two.

We are then under the hypotheses of \cite[Theorem 7]{kato53} (cf.\ \cite{kato76}) applied to the operator associated with the bilinear form
\begin{equation}
\label{qtilde}
\widetilde{\mathcal Q}:=L+\frac1\alpha G,
\end{equation}
with the same form domains as described above; in particular, the differential problem associated with the bilinear form $\widetilde{\mathcal Q}$ has the same boundary conditions as for $\mathcal Q$, while the equation (which is the same in all cases) becomes
$$
\frac1\alpha\Delta^2 u-\Delta u= \frac\lambda\alpha u.
$$
The claimed convergence as $\alpha \to +\infty$ now follows from \cite[Theorem 7]{kato53}.
\end{proof}

\begin{remark}
We note that \cite[Theorem 8]{kato53} and, as a complement, also \cite[Theorem 2.4]{gre69} (see also \cite{gre81}) provide a further asymptotic expansion of the eigenvalues in terms of $\alpha$ and also provide estimates for the rate of convergence of the eigenfunctions. However, in order to get this type of additional information one would need the stronger hypothesis that the relevant Laplacian eigenfunction is in the form domain of $G$. In order to have this condition satisfied in the Full Robin case, it suffices for instance that the domain $\Omega$ is smooth enough (e.g., $\Omega\in C^2$); however, the Dirichlet Bilaplacian cannot be treated within this framework since no Laplacian eigenfunction can be in $H^2_0(\Omega)$. At any rate, the additional information we obtain if the condition is satisfied includes in particular an estimate of the type
\begin{equation}
\label{eq:asympt_eigenv}
\lambda_k(\alpha)=\alpha \mu_k^{\sharp}+\lambda_k'+\mathcal{O}(\alpha^{-1}),
\end{equation}
where $\mu_k^{\sharp}$ is either $\mu_k^{D}$ or $\mu_k^{N}$, and $\lambda_k'$ can be computed explicitly in terms of the eigenspace associated with $\mu_k^{\sharp}$.
\end{remark}

\begin{problem}
Study asymptotic expansions of the type of \eqref{eq:asympt_eigenv} in more detail for any of the biharmonic Robin problem.
\end{problem}

In Theorem \ref{conv-alpha} the Robin parameters are assumed to be constant (or bounded), and this results in their disappearance in the limit. On the other hand, if they are allowed to diverge to infinity at the same time that $\alpha \to +\infty$, then the analysis of the behaviour of the eigenvalues may become more involved (cf.\ Remark \ref{rem:interplay}). However, in specific cases we can still recover some Robin parameter in the limit. Here we will consider the rather natural case of a linear relationship, $\gamma = \alpha\tilde\gamma$ for some fixed $\tilde\gamma \in \R$ with $\beta$ bounded, and, separately $\beta = \alpha\tilde\beta$ for some fixed $\tilde\beta \in \R$ with $\gamma$ bounded.

\medskip
\textbf{The case of divergent $\gamma$.} We first treat the former case, which is much simpler: we will, after a suitable renormalisation, recover the Robin Laplacian \eqref{eq:robin-laplacian}.

\begin{theorem}
\label{thm:real-robin-convergence}
Let $\beta$ be bounded, and let $\gamma=\alpha\tilde\gamma$ for some fixed $\tilde\gamma \in \R$, and write $\lambda_k(\alpha,\gamma)$ for the $k$-th eigenvalue of either the Full Robin problem \eqref{robin1strong} or the KS problem \eqref{kuttler-sigillito}. Then
\begin{equation*}
\lim_{\alpha\to+\infty}\frac{\lambda_k(\alpha,\gamma)}{\alpha}=\mu^R_k (\tilde\gamma),
\end{equation*}
with convergence of the corresponding eigenprojections in $L^2(\Omega)$.
\end{theorem}

\begin{proof}
The proof is based on the same argument as that of Theorem \ref{conv-alpha}, by redefining the operators $L$ and $G$ in \eqref{qtilde}. More precisely, we take
\begin{displaymath}
	L_{\tilde\gamma}(u,v) := \int_\Omega \nabla u \cdot \nabla v\,dx + \tilde\gamma \int_{\partial\Omega} uv\,d\mathcal{H}^{d-1}(x)
\end{displaymath}
and $G=\mathcal{Q}-\alpha L_{\tilde\gamma}$, both defined originally on $H^2(\Omega)$, so that the Friedrichs extension of the operator associated with $L$ is the Robin Laplacian with form domain $H^1(\Omega)$. We may then imitate the proof of Theorem~\ref{conv-alpha} exactly.
\end{proof}

\begin{remark}
\label{rem:alpha-nonlinear}
If instead of positing a linear relationship $\gamma = \alpha \tilde\gamma$, we assume some more complicated behaviour $\gamma = f(\alpha)$, then clearly other limit problems will be possible; for example, if $\gamma$ grows more rapidly than $\alpha$ (but $\beta$ remains bounded), then after renormalisation we may expect convergence to the eigenvalues of the Dirichlet Laplacian. We will not go into details here, but merely observe that the conclusion of Theorem \ref{thm:real-robin-convergence} continues to hold under the weaker assumption that $\gamma/\alpha\to\tilde{\gamma}$.
\end{remark}

\begin{problem}
Determine $\lim_{\alpha\to+\infty}\frac{\lambda_k(\alpha,\gamma)}{\alpha}$ under other assumptions on the relationship $\gamma = f(\alpha)$.
\end{problem}

\medskip
\textbf{The case of divergent $\beta$.} The second case, when $\beta \to +\infty$ with $\alpha$ but $\gamma$ remains bounded, is far more interesting: after renormalisation we obtain a kind of Laplacian which to the best of our knowledge has not been previously studied, where in a very particular sense there is a ``decoupling'' between the dynamics on $\Omega$ and on its boundary. In the case of sufficiently smooth boundary this actually reduces to the Neumann Laplacian, but in the less smooth case it will be more complicated; it may be studied within the framework of the \emph{$j$-elliptic operators} originally introduced in \cite{ate11,ate12} to extend the ``French approach'' to differential operators via sectorial forms (as exemplified by Lions, see \cite{lions}) to the case where the form domain is not necessarily embedded in the ambient Hilbert space (see in particular \cite[Section~2]{ate12}).

Here we will first introduce the operator at the form level and then afterwards study the case of smooth boundary, where the Total Trace Theorem (see e.g., \cite[Theorem 3.4]{laproz}) can be used to show that the associated operator is exactly the Neumann Laplacian. We will finish by showing that there really is convergence of the Robin Bilaplacian eigenvalues and eigenfunctions (in the usual sense) to the ones of this operator.

We consider the product space $H^1(\Omega) \times L^2 (\partial\Omega)$ equipped with its canonical norm and, given the Robin-type parameter $\tilde\beta \in \R$, define the quadratic form $L_{\tilde\beta}: V \times V \to \R$ by
\begin{equation}
\label{eq:weird-form}
	L_{\tilde\beta} \left((u,f),(v,g)\right) = \int_{\Omega}\nabla u\cdot \nabla v\, dx+\tilde\beta\int_{\partial\Omega}fg \,d \mathcal{H}^{d-1}(x)
\end{equation}
on the closed subspace
\begin{equation}
\label{eq:weird-form-domain}
\begin{split}
	V:= \Big\{(u,f) \in H^1(\Omega) \times L^2(\partial\Omega): \exists \varphi_n \in H^2(\Omega) \text{ such that}\\
	\varphi_n \to u \text{ in }H^1(\Omega) \text{ and }
		\frac{\partial\varphi_n}{\partial\nu} \to f \text{ in } L^2(\partial\Omega) \Big\}
\end{split}
\end{equation}
of $H^1 (\Omega) \times L^2 (\partial\Omega)$, which we recognise to be the natural (closure of the) form domain associated with the limit operator we expect to obtain in this case as $\alpha \to +\infty$, as defined and described above. We take the canonical norm on $H^1(\Omega) \times L^2(\partial\Omega)$ as our norm on $V$, which we shall denote by $\|\cdot\|_V$. We observe that, if $\Omega \in C^{2,1}$, then $V = H^1 (\Omega) \times L^2 (\partial\Omega)$ by the aforementioned Total Trace Theorem \cite[Theorem 3.4]{laproz}, whereas for general Lipschitz $\Omega$, $V$ will in general be a proper subset of $H^1(\Omega) \times L^2(\partial\Omega)$.

We shall denote by $A_{\tilde\beta}$ the operator on $L^2(\Omega)$ associated with the form $L_{\tilde\beta}$ in the sense of \cite[Theorem~2.1]{ate12}. That is, since the natural limit form domain $V$ is not a subset of $L^2(\Omega)$, we introduce the natural mapping $j : V \to L^2 (\Omega)$ given by $(u,f) \mapsto u$, which for bounded Lipschitz $\Omega$ is immediately seen to be compact and linear but not an injection ($j$ is the composition of the projection onto $H^1(\Omega)$ and the embedding of $H^1(\Omega)$ into $L^2(\Omega)$). We then define the operator $A_{\tilde \beta}: D(A_{\tilde \beta}) \subset L^2 (\Omega) \to L^2 (\Omega)$ associated with the \emph{pair} $(L_{\tilde\beta},j)$, that is, defined by the rule that $u \in D(A_{\tilde\beta})$ and $A_{\tilde\beta}u = h$ if and only if there exists an element $(u,f) \in V$ such that $j(u,f) = u$ and
\begin{equation}
\label{eq:what-form}
	L_{\tilde\beta} \left((u,f),(v,g)\right) = ( h, j(v,g) )_{L^2(\Omega)}
\end{equation}
for all $(v,g) \in V$. That $A_{\tilde\beta}$ admits a discrete spectrum, at least when $\tilde\beta > 0$, is a consequence of the general theory of $j$-elliptic forms:

\begin{lemma}
Suppose that $\tilde\beta > 0$. Then the operator $A_{\tilde\beta}$ defined above is self-adjoint with compact resolvent and semi-bounded from below. In particular, its spectrum takes the form of an increasing sequence of eigenvalues $\mu^S_1 (\tilde\beta) \leq \mu^S_2 (\tilde\beta) \leq \ldots \to +\infty$.
\end{lemma}

The case $\tilde\beta \leq 0$ would require a much more extensive analysis, which we will not perform here (cf.\ Open Problem~\ref{prob:neg-tildebeta}). However, in the smooth case we will recover the conclusion of the lemma for any $\tilde\beta \in \R$ from the fact that $A_{\tilde\beta}$ is just the Neumann Laplacian (see Proposition~\ref{prop:weird-neumann}).

\begin{proof}
Fix $\tilde\beta > 0$. We first observe that by definition $L_{\tilde\beta}$ satisfies the following \emph{$j$-ellipticity estimate}:
\begin{equation}
\label{eq:j-ellipticity}
	L_{\tilde\beta} \left((u,f),(u,f)\right) + \|j(u,f)\|^2_{L^2(\Omega)} \geq \min \{1, \tilde \beta\} \|(u,f)\|^2_V
\end{equation}
for all $(u,f) \in V$. Since $j$ is compact, it follows from \cite[Lemma~2.7]{ate12} that $A_{\tilde\beta}$ has compact resolvent.

For the self-adjointness, note that $\{u \in H^1(\Omega): (u,f) \in V\}$ is certainly dense in $H^1(\Omega)$ as by choice of $V$ it contains a copy of $H^2(\Omega)$. It follows that $j(V)$ is dense in $L^2(\Omega)$. Since $L_{\tilde\beta}$ is symmetric, \cite[Remark~3.5]{ate12} implies that $A$ is self-adjoint.
\end{proof}

\begin{proposition}
\label{prop:weird-neumann}
Suppose $\Omega$ is of class $C^{2,1}$. Then, for any $\tilde\beta \in \R$, the operator $A_{\tilde\beta}$ on $L^2(\Omega)$ described above, associated with the form $L_{\tilde{\beta}}$ on $V$, is the Neumann Laplacian, given by
\begin{displaymath}
\begin{aligned}
	D(A_{\tilde{\beta}}) &= \left\{u \in H^1(\Omega): \Delta u \in L^2(\Omega),\, \frac{\partial u}{\partial\nu} = 0\right\},\\
	A_{\tilde{\beta}}u &= -\Delta u,
\end{aligned}
\end{displaymath}
where $\Delta u$ and $\frac{\partial u}{\partial\nu}$ are to be interpreted in the usual distributional sense.
\end{proposition}

In particular, in this case, for all $\tilde\beta \in \R$ the operator $A_{\tilde{\beta}}$ is self-adjoint with compact resolvent, and we trivially have $\mu_k^S (\tilde{\beta}) = \mu_k^N$ for all $k\geq 1$.

\begin{proof}
We start by recalling that in this case the form domain $V$ introduced in \eqref{eq:weird-form-domain} is equal to the whole of $H^1(\Omega) \times L^2(\partial\Omega)$.

We now fix $u \in D(A_{\tilde\beta})$ and suppose that $A_{\tilde{\beta}}u = h \in L^2(\Omega)$. We first observe that $h=-\Delta u$ in $L^2(\Omega)$; indeed, taking $v \in H^1_0 (\Omega)$ and $g = 0$ in \eqref{eq:what-form}, we obtain
\begin{displaymath}
	\int_\Omega \nabla u \cdot \nabla v\, dx = \int_\Omega hv\,dx
\end{displaymath}
for all $v \in H^1_0(\Omega)$, whence the claim. To show that $\frac{\partial u}{\partial\nu} = 0$ distributionally, we claim that \eqref{eq:what-form} holds for the element $(u,0) \in V =  H^1(\Omega) \times L^2(\partial\Omega)$ (note that in general this will only be an element of $V$ in the smooth case). Then \eqref{eq:what-form} means exactly that
\begin{equation}
\label{eq:neumann-identity}
	\int_\Omega \nabla u\cdot \nabla v\,dx = -\int_\Omega \Delta u v\,dx
\end{equation}
for all $v \in H^1(\Omega)$. But we now recall the variational definition of $\frac{\partial u}{\partial \nu}$: this is, by definition, the function $\varphi$ in $L^2(\partial\Omega)$, if it exists, such that
\begin{displaymath}
	\int_\Omega \nabla u\cdot \nabla v\,dx + \int_{\partial\Omega} \varphi v\,d\mathcal{H}^{d-1}(x) = -\int_\Omega \Delta u v\,dx
\end{displaymath}
for all $v \in H^1(\Omega)$. We thus see that indeed $\frac{\partial u}{\partial \nu} = 0$.

Conversely, suppose $u \in H^1(\Omega)$ satisfies $\Delta u \in L^2(\Omega)$ and $\frac{\partial u}{\partial\nu} = 0$; we wish to show that $u \in D(A_{\tilde{\beta}})$ and $A_{\tilde{\beta}}u = -\Delta u$. Now by assumption and the variational definition of $\Delta u$, $u$ satisfies \eqref{eq:neumann-identity} for all $v \in H^1(\Omega)$. But this says exactly that
\begin{displaymath}
	L_{\tilde{\beta}} ((u,0),(v,g)) = (-\Delta u,j(v,g))_{L^2(\Omega)}
\end{displaymath}
for all $(v,g) \in V$; hence, by definition (cf.\ \eqref{eq:what-form}), $u \in D(A_{\tilde{\beta}})$ and $A_{\tilde{\beta}}u = -\Delta u$.
\end{proof}

We can finally turn to our main convergence result.

\begin{theorem}
\label{thm:strange-robin}
Let $\gamma$ be bounded, and let $\beta=\alpha\tilde\beta$, where we assume either that $\tilde\beta > 0$ or that $\Omega \in C^{2,1}$ and $\tilde\beta \in \R$. Also let $\lambda_k(\alpha,\beta)$ be the $k$-th eigenvalue of either the Full Robin problem \eqref{robin1strong} or the NR problem \eqref{navier-robin}. Then
\begin{equation*}
\lim_{\alpha\to+\infty}\frac{\lambda_k(\alpha,\beta)}{\alpha}=\mu^S_k (\tilde\beta),
\end{equation*}
with convergence of the corresponding eigenprojections in $L^2(\Omega)$.
\end{theorem}

\begin{proof}
The proof is again based on the same argument as in Theorem \ref{conv-alpha}. Let $\pi: H^2(\Omega)\to V\subseteq H^1(\Omega)\times L^2(\partial\Omega)$ be the standard projection $\pi(u)=(u,\partial_\nu u)$. We set $L^0_{\tilde{\beta}}(u,v):=L_{\tilde{\beta}}(\pi(u),\pi(v)) $, where $L_{\tilde\beta}$ is as in \eqref{eq:weird-form}, and $G=\mathcal{Q}-\alpha L^0_{\tilde\beta}$ on $H^2(\Omega)$. In this case, the closure of $H^2(\Omega)$ with respect to the form $L^0_{\tilde\beta} + \omega (\cdot,\cdot)_{L^2(\Omega)}$ (for sufficiently large $\omega \geq 0$) is immediately seen to be the space $V$ from \eqref{eq:weird-form-domain}, meaning that $L_{\tilde{\beta}}$ is the Friedrichs extension of $L^0_{\tilde{\beta}}$. Note also that all operators are defined on the same space $L^2(\Omega)$. The proof of Theorem~\ref{conv-alpha} may now be repeated verbatim, invoking \cite[Theorem 7]{kato53} to prove the claimed convergence as $\alpha \to +\infty$.
\end{proof}

\begin{problem}
\label{prob:neg-tildebeta}
Investigate the operator $A_{\tilde\beta}$, and hence the limit in Theorem~\ref{thm:strange-robin}, in the non-smooth case, $\Omega \not\in C^{2,1}$, if $\tilde\beta \leq 0$.
\end{problem}

In this context we observe that for negative $\tilde\beta$ the ellipticity estimate \eqref{eq:j-ellipticity} breaks down irreparably; indeed, for $(u,f) \in V$ the functions $u$ and $f$ are no longer completely decoupled, as there must be an interaction at all points where the boundary is not locally smooth. In this case it is no longer clear \emph{a priori} whether $A_{\tilde\beta}$ is bounded below, or has compact resolvent; understanding this operator would require a more careful study of the form domain $V$ and thus the behaviour of higher-order traces on less smooth domains. It would be particularly interesting to understand how the eigenvalues of the Robin Bilaplacian behave in the corresponding limit; however, we will not explore this question further here.

\begin{remark}
Analogously to Remark~\ref{rem:alpha-nonlinear}, Theorem~\ref{thm:strange-robin} continues to hold if $\gamma$ remains bounded and $\beta/\alpha$ only behaves asymptotically like $\tilde{\beta}$.
\end{remark}

\begin{problem}
Determine $\lim_{\alpha\to+\infty}\frac{\lambda_k(\alpha,\beta)}{\alpha}$ in the case of a more general relationship $\beta = f(\alpha)$. Positing a power relationship $\beta \sim \alpha^\theta$, for what values of $\theta \geq 0$ do we have convergence to the Neumann Laplacian in the smooth case?
\end{problem}


\section{Divergence for large negative values of the parameters}
\label{sec:num-range}

In this section we analyse the behaviour of the eigenvalues when either of the parameters $\alpha$, $\beta$, or $\gamma$ approach minus infinity, while the others stay fixed (or bounded). Note that, since in this regime the eigenvalues diverge in a different manner depending on the diverging parameter, the possibility of ``joint divergence'' may lead to very intricate behaviour.

\medskip
\textbf{Main divergence results.}
We start with the following  basic result, which states that if any parameter diverges to $-\infty$, then there is a sequence of divergent eigenvalues.
We recall the convention that $\frac{1}{\infty} = 0$ and also take the standard conventions $0 \cdot \infty = 0$, $\min \{0,+\infty \} = 0$ etc.

\begin{theorem}
\label{thm:steklovbuckling}
Let $\Omega \subset \mathbb R^d$ be a bounded Lipschitz domain and suppose that $\sigma \in (-\frac{1}{d-1},1)$ is fixed.
\begin{itemize}
\item[(a)] For any $\beta,\gamma \in \R \cup \{+\infty\}$ and for any $k\in\mathbb N$, $\lambda_k (\Omega,\sigma,\alpha,\beta,\gamma)\to-\infty$ as $\alpha\to-\infty$.
\item[(b)] For any $\alpha \in \R$, for any $\gamma \in \R \cup \{+\infty\}$, and for any $k\in\mathbb N$, $\lambda_k (\Omega,\sigma,\alpha,\beta,\gamma)\to-\infty$ as $\beta\to-\infty$.
\item[(c)]  For any $\alpha \in \R$, for any $\beta\in \R \cup \{+\infty\}$, and for any $k\in\mathbb N$, $\lambda_k (\Omega,\sigma,\alpha,\beta,\gamma)\to-\infty$ as $\gamma\to-\infty$.
\end{itemize}
\end{theorem}

Theorem \ref{thm:steklovbuckling} is valid under minimal regularity assumptions as it can be proved using properties of the associated Steklov eigenvalue problems as well as buckling-type eigenvalue problems. We provide a simple proof below.

We observe that Theorem~\ref{thm:steklovbuckling} gives no information about the rate of divergence of the eigenvalues, which we obtain by assuming that the domain $\Omega$ has some more regularity. Nevertheless, we expect similar rates to hold in general (cf.\ Open Problems \ref{prob:neg-div-more-precise}-\ref{prob:simple-study}).

\begin{theorem}
\label{thm:divergence-basic}
Let $\Omega \subset \mathbb R^d$ be a bounded $C^1$-domain and suppose that $\sigma \in (-\frac{1}{d-1},1)$ is fixed. Then, for each fixed $k \in \N$, and in each case with the respective other variables bounded,
\begin{itemize}
\item[(a)] $\displaystyle \lambda_k (\Omega,\sigma,\alpha,\beta,\gamma) \asymp -|\alpha|^2 \quad \text{as } \alpha \to -\infty$;
\item[(b)] $\displaystyle \lambda_k (\Omega,\sigma,\alpha,\beta,\gamma) \asymp -|\beta|^4 \quad \text{as } \beta \to -\infty$;
\item[(c)] $\displaystyle \lambda_k (\Omega,\sigma,\alpha,\beta,\gamma) \asymp -|\gamma|^{4/3} \quad \text{as } \gamma \to -\infty$.
\end{itemize}
In each case, for $k=1$ the limit is uniform in the other two variables, as long as these each remain within a compact interval.
\end{theorem}

Here we have written $f \asymp g$ to mean that $f = \mathcal{O}(g)$ and $g = \mathcal{O}(f)$. Theorem~\ref{thm:divergence-basic} will be obtained from separate results controlling the rate of divergence of the eigenvalues from below (Theorem~\ref{thm:num-range-eig-div}; see also Lemma~\ref{lem:numerical-range} on the numerical range of the form) and from above (Theorem~\ref{thm:divergence-1}, see also the remarks after it).

We observe that the estimates from below that we obtain are always uniform with respect to the other parameters, as long as these remain bounded from below. However, the technique used to prove Theorem~\ref{thm:divergence-1} provides uniform estimates only for the first eigenvalue. We suspect that this difference may be due to the argument rather than the actual behaviour of the eigenvalues, and that a more sophisticated but more involved technique such as based on Dirichlet-Neumann bracketing (see below and cf., e.g., \cite{pankpop}) might be expected to give uniform estimates.

\begin{problem}
\label{prob:unif-div}
Study in more detail the limits in Theorem~\ref{thm:divergence-basic}, also using alternative approaches, to understand whether they are uniform in the other variables for $k>1$.
\end{problem}

\begin{problem}
\label{prob:neg-div-more-precise}
Determine the coefficients of the respective leading terms in the asymptotic expansion of $\lambda_k$ in the different parameter regimes in the setting of Theorem~\ref{thm:divergence-basic}, not just for $C^1$-domains but for domains with corners and general Lipschitz domains. We observe that Theorem~\ref{thm:divergence-1} provides one-sided estimates, see also Remark~\ref{rem:big-neg-alpha} for case (a).
\end{problem}

As in the case of the Laplacian, it is natural to expect that the coefficient of the leading term will depend on smoothness properties of the boundary in cases (b) and (c); in the case of the Laplacian this has been a topic of active research in recent years (see, e.g., \cite{khalile18,lp08} and the references therein).

\begin{problem}
\label{prob:neg-div-more-precise2}
In the case where $\partial\Omega$ is smooth, obtain more terms in the asymptotic expansion for $\lambda_k$ in the setting of Theorem~\ref{thm:divergence-basic}, in the various parameter regimes.
\end{problem}

The most natural approach in the smooth case would be to develop a Dirichlet-Neumann bracketing technique for the Bilaplacian, which is currently completely open. However, as a first step, it may be advisable to:

\begin{problem}
\label{prob:simple-study}
Study the divergence question in model cases where the eigenvalues can be described more or less explicitly: half-lines/half-spaces; balls; annuli; rectangles.
\end{problem}

Finally, let us also observe that Theorem~\ref{thm:divergence-basic} does not say anything in the case where either $\beta$ or $\gamma$ are plus infinity, yet we conjecture the same rate of divergence to hold anyway (cf.\ Remark~\ref{rem:intermediate-regime} and Open Problem~\ref{op:divergence-rates}).

\begin{proof}[Proof of Theorem~\ref{thm:steklovbuckling}]
Throughout this proof we will write $\lambda_k(\alpha,\beta,\gamma)$ in place of $\lambda_k (\Omega,\sigma,\alpha,\beta,\gamma)$, as $\Omega$ and $\sigma$ will be treated as fixed.

For point (a), we consider the following buckling-type eigenvalue problem
\begin{equation}
\label{bucklingstrong}
\begin{cases}
\Delta^2u=-\Lambda\Delta u, & \text{in\ }\Omega,\\
\Beta(u) = -\beta\frac{\partial u}{\partial\nu}, & \text{on\ }\partial\Omega,\\
\Gamma(u)=-\gamma u, & \text{on\ }\partial\Omega,
\end{cases}
\end{equation}
whose weak formulation is
\begin{equation}
\int_{\Omega} \left((1-\sigma)D^2u:D^2 v +\sigma\Delta u\Delta v \right)\,dx
+\beta \int_{\partial\Omega}\frac{\partial u}{\partial\nu}\frac{\partial  v }{\partial\nu} \,d\mathcal{H}^{d-1}(x)
+\gamma \int_{\partial\Omega} u  v  \,d\mathcal{H}^{d-1}(x)
=\Lambda\int_{\Omega} \nabla u\cdot \nabla v  \,dx
\end{equation}
for $u,v\in H^2(\Omega)$. If either $\beta$ or $\gamma$ equals plus infinity, we can slightly modify problem \eqref{bucklingstrong} in the same way as we do for problem \eqref{robinstrong} (see Section \ref{sec:setting}). In any case, problem \eqref{bucklingstrong} has a non-decreasing diverging sequence of eigenvalues of finite multiplicity
$$
\Lambda_1(\beta,\gamma)\le \Lambda_2(\beta,\gamma)\le\dots\le\Lambda_k(\beta,\gamma)\le\dots\to+\infty.
$$
This fact can be easily proved following the ideas presented in, e.g., \cite{bula2013}. Now, if we set $\alpha=-\Lambda_k=-\Lambda_k(\beta,\gamma)$, then
$$
\lambda_k(-\Lambda_k,\beta,\gamma)=0
$$ 
and the multiplicity of $\lambda_k(-\Lambda_k,\beta,\gamma)$ is the same as that of $\Lambda_k$. Now consider the first $k$ eigenfunctions of problem \eqref{bucklingstrong} $v_1,\dots,v_k$, and let $v\in V_k=\Span \{ v_1,\dots,v_k\}$. Then
$$
\mathcal Q_{\Omega,\sigma,\alpha,\beta,\gamma}(v, v)\le \Lambda_k\int_{\Omega} |\nabla v|^2\,dx +\alpha \int_{\Omega} |\nabla v|^2\,dx,
$$
which implies that
\begin{equation}
\label{ineqbuckling}
\lambda_k(\alpha,\beta,\gamma)\le(\Lambda_k+\alpha)\max_{v\in V_k}\frac{\int_{\Omega} |\nabla v|^2\,dx}{\int_{\Omega} |v|^2\,dx}.
\end{equation}
Letting $\alpha\to-\infty$, the right hand side in \eqref{ineqbuckling} goes to minus infinity, proving the claim.

We now pass to point (b). In this case, we introduce the following Steklov-type problem (cf.\ \cite{lapro})

\begin{equation}
\label{steklov1}
\begin{cases}
\Delta^2u-\alpha\Delta u=0, & \text{in\ }\Omega,\\
\Beta(u) = \eta\frac{\partial u}{\partial\nu}, & \text{on\ }\partial\Omega,\\
\Gamma(u)=-\gamma u, & \text{on\ }\partial\Omega,
\end{cases}
\end{equation}
whose weak formulation is
\begin{equation}
\int_{\Omega} \left((1-\sigma)D^2u:D^2 v +\sigma\Delta u\Delta v \right)\,dx +\alpha \int_{\Omega} \nabla u\cdot \nabla v  \,dx
+\gamma \int_{\partial\Omega} u  v  \,d\mathcal{H}^{d-1}(x)
=\eta \int_{\partial\Omega}\frac{\partial u}{\partial\nu}\frac{\partial  v }{\partial\nu} \,d\mathcal{H}^{d-1}(x)
\end{equation}
where $u,v\in H^2(\Omega)$. As in the previous case, if $\gamma$ equals plus infinity, we can slightly modify problem \eqref{steklov1} in the same way as we do for problem \eqref{robinstrong}. In any case, problem \eqref{steklov1} has a non-decreasing diverging sequence of eigenvalues of finite multiplicity (see \cite{lapro})
$$
\eta_1(\alpha,\gamma)\le \eta_2(\alpha,\gamma)\le\dots\le\eta_k(\alpha,\gamma)\le\dots\to+\infty.
$$
Now, if we set $\beta=-\eta_k=-\eta_k(\alpha,\gamma)$, then
$$
\lambda_k(\alpha,-\eta_k,\gamma)=0
$$ 
and the multiplicity of $\lambda_k(\alpha,-\eta_k,\gamma)$ is the same of $\eta_k$. The same test-function argument used in point a) allows to conclude.

Point (c) can be treated in the same way, considering the Steklov-type problem
\begin{equation}
\label{steklov2}
\begin{cases}
\Delta^2u-\alpha\Delta u=0, & \text{in\ }\Omega,\\
\Beta(u) = -\beta\frac{\partial u}{\partial\nu}, & \text{on\ }\partial\Omega,\\
\Gamma(u)=\xi u, & \text{on\ }\partial\Omega;
\end{cases}
\end{equation}
we refer to \cite{lapro} for more details on this problem.
\end{proof}


\medskip
\textbf{Estimates from below on the rate of divergence.} We next establish the maximal possible rate of divergence of the eigenvalues. We start with the following estimate on the quadratic form, which will then directly imply the bounds on the eigenvalues as $\alpha$, $\beta$ and/or $\gamma$ diverge to $-\infty$ (Theorem~\ref{thm:num-range-eig-div}). This also yields the first half of Theorem~\ref{thm:divergence-basic}.

\begin{lemma}
\label{lem:numerical-range}
Let $\Omega \subset \mathbb R^d$ be a bounded Lipschitz domain and let $\alpha\in \R$, $\beta,\gamma \in \R \cup \{+\infty\}$ and $\sigma \in (-\frac{1}{d-1},1)$ be given. Then there exist constants $C_1,\ldots,C_4 > 0$ depending only on $\Omega$ such that
\begin{equation}
\begin{split}
\label{numerical-range}
	\mathcal{Q}_{\Omega,\sigma,\alpha,\beta,\gamma} (u,u) \geq \|D^2u\|_2^2 + &\alpha \|\nabla u\|_2^2 + 
	\min\{0,\beta\} \left(C_1 \|D^2u\|_2^{3/2} + C_2\|D^2u\|_2 \right)\\ &\qquad + \min\{0,\gamma\} \left(C_3\|D^2u\|_2^{1/2}+C_4\right)
\end{split}
\end{equation}
for all $u$ in the form domain of $\mathcal{Q}_{\Omega,\sigma,\alpha,\beta,\gamma}$ for which $\|u\|_2=1$.
\end{lemma}

We postpone the proof of Lemma~\ref{lem:numerical-range} in order to give our principal estimate from below on the rate of divergence of the eigenvalues.

\begin{theorem}
\label{thm:num-range-eig-div}
Let $\Omega \subset \mathbb R^d$ be a bounded Lipschitz domain and suppose that $\sigma \in (-\frac{1}{d-1},1)$ is fixed.
\begin{itemize}
\item[(a)] If $\beta,\gamma \in \R \cup \{+\infty\}$ are bounded from below, then
\begin{displaymath}
	\limsup_{\alpha\to -\infty} \frac{\lambda_1 (\Omega,\sigma,\alpha,\beta,\gamma)}{-|\alpha|^2} < +\infty.
\end{displaymath}
\item[(b)] If $\alpha \in \R$ is bounded and $\gamma \in \R \cup \{+\infty\}$ is bounded from below, then
\begin{displaymath}
	\limsup_{\beta\to -\infty} \frac{\lambda_1 (\Omega,\sigma,\alpha,\beta,\gamma)}{-|\beta|^4} < +\infty.
\end{displaymath}
\item[(c)]  If $\alpha \in \R$ is bounded and $\beta\in \R \cup \{+\infty\}$ is bounded from below, then
\begin{displaymath}
	\limsup_{\gamma\to -\infty} \frac{\lambda_1 (\Omega,\sigma,\alpha,\beta,\gamma)}{-|\gamma|^{4/3}} < +\infty.
\end{displaymath}
\end{itemize}
In each case, treating $\Omega$ and $\sigma$ as fixed, both the limit superior and the rate of convergence to it can be controlled purely in terms of the respective lower bounds on the other two parameters.
\end{theorem}

\begin{proof}
In each case, the idea is to use \eqref{numerical-range} to show that no rate of divergence of $\|D^2u\|_2$ as the parameter in question tends to $-\infty$ can lead to $\mathcal{Q}_{\Omega,\sigma,\alpha,\beta,\gamma} (u,u)$ diverging more rapidly than the power in question.

Concretely, for (a) we first invoke the estimate
\begin{equation}
\label{nablacontrol}
	\|\nabla u\|_2 \leq C \|D^2u\|_2^{1/2}\|u\|_2^{1/2}
\end{equation}
for all $u \in H^2(\Omega)$ and in particular for all $u$ in the form domain of $\mathcal{Q}_{\Omega,\sigma,\alpha,\beta,\gamma}$, where $C>0$ is a constant depending only on $\Omega$ (see, e.g., \cite[Section 4.4]{burenkov} or \cite{bin}). 
Assuming without loss of generality that $\alpha < 0$, we see that \eqref{numerical-range} reduces to
\begin{equation}
\label{eq:my-qest}
\begin{split}
	\mathcal{Q}_{\Omega,\sigma,\alpha,\beta,\gamma} (u,u) \geq \|D^2u\|_2^2 + &\alpha C\|D^2 u\|_2 + 
	\min\{0,\beta\} \left(C_1 \|D^2u\|_2^{3/2} + C_2\|D^2u\|_2 \right)\\ &\qquad + \min\{0,\gamma\} \left(C_3\|D^2u\|_2^{1/2}+C_4\right).
\end{split}
\end{equation}
Now we suppose that $\alpha \to -\infty$, while we assume that $\beta,\gamma$ take on a fixed, negative value, which we may take to be any bound from below on their range of values, as this can only worsen the estimate \eqref{eq:my-qest} (obviously, if they are positive, we can estimate the respective terms from below by zero). Denote by $u_1 = u_1 (\alpha)$ an eigenfunction associated with $\lambda_1$, normalised to have unit $L^2(\Omega)$-norm. If $\|D^2 u_1(\alpha)\|_2$ remains bounded as $\alpha \to -\infty$ (or just for a sequence of values of $\alpha$), then
\begin{displaymath}
	\mathcal{Q}_{\Omega,\sigma,\alpha,\beta,\gamma} (u_1(\alpha),u_1(\alpha)) \geq c_1 \alpha + c_2
\end{displaymath}
for all $\alpha < 0$ and for certain constants $c_1,c_2>0$ which in particular implies (a) (at least for this sequence). Now suppose $\|D^2 u_1 (\alpha)\|_2$ diverges (again, possibly for a sequence), which means that
\begin{equation*}
	\liminf_{n\to\infty} \frac{\|D^2u_1(\alpha)\|_2^2}{|\beta|\left(C_1 \|D^2u_1(\alpha)\|_2^{3/2} + C_2\|D^2u_1(\alpha)\|_2 \right)}
	=	\liminf_{n\to\infty} \frac{\|D^2u_1(\alpha)\|_2^2}{|\gamma|\left(C_3\|D^2u_1(\alpha)\|_2^{1/2}+C_4\right)} = +\infty,
\end{equation*}
so that we may neglect the terms involving $\beta$ and $\gamma$ when calculating limits. Thus, setting $x:= \|D^2u_1(\alpha)\|_2$, it follows from \eqref{eq:my-qest} and a standard calculus argument that
\begin{equation}
\label{eq:estimate_uniform}
	\lambda_1 \geq x^2 + C\alpha x + \tilde C_1 x^{\frac 3 2} +\tilde C_2
\end{equation}
for all $\alpha < 0$, where $\tilde C_i=\tilde C_i(\Omega,\beta,\gamma)$, $i=1,2$ (or, more precisely, $\tilde C_i = \tilde C_i(\Omega,\min \beta,\min\gamma)$). Another standard calculus argument allows us to minimise the expression in the right-hand side, leading to $\lambda_1 \geq -C^2|\alpha|^2/4 + \mathcal{O}(|\alpha|^{3/2})$ as $\alpha \to -\infty$, uniformly in $\beta$ and $\gamma$ bounded from below. This implies (a).

The arguments for (b) and (c) are entirely analogous and we do not go into details; we merely observe that the critical rate of growth of $\|D^2u_1\|_2$ (i.e., up to a constant, the correct choice of $x$ after the lower order terms have been discarded, which leads to the fastest possible rate of divergence of the right-hand side of \eqref{numerical-range}) is $|\beta|^{1/2}$ in case (b) and $|\gamma|^{2/3}$ in case (c).
\end{proof}

The proof of Lemma~\ref{lem:numerical-range} is, in turn, an immediate consequence of the following two lemmata, which control the $L^2(\partial\Omega)$-norm of the normal derivative and the trace of a function $u \in H^2(\Omega)$.

\begin{lemma}
Let $\Omega \subset \mathbb R^d$ be a bounded Lipschitz domain. There exist constants $C_1,C_2 > 0$ such that
\begin{displaymath}
	\int_{\partial\Omega} \left|\frac{\partial u}{\partial\nu}\right|^2\,d\mathcal{H}^{d-1}(x) \leq C_1\|D^2u\|_2^{3/2}+C_2\|D^2u\|_2
\end{displaymath}
for all $u \in H^2 (\Omega)$ with $\|u\|_2 = 1$.
\end{lemma}

\begin{proof}
In order to simplify the notation, in this proof we will write $C>0$ for a constant depending only $\Omega$ but which may change from line to line and even from term to term. We show that
\begin{equation}
\label{normalderivintermedest}
	\int_{\partial\Omega} \left|\frac{\partial u}{\partial\nu}\right|^2\,d\mathcal{H}^{d-1}(x) \leq C\|\nabla u\|_2\|D^2u\|_2 + C\|\nabla u\|_2^2
\end{equation}
for all $u \in H^2 (\Omega)$; applying \eqref{nablacontrol} to \eqref{normalderivintermedest} then yields the conclusion of the lemma. To prove \eqref{normalderivintermedest}, we adapt the proof of \cite[Lemma~6.4]{bkl19}. We fix an arbitrary point $z \in \partial\Omega$, an open neighbourhood $U_z \subset \R^d$ of $z$ and a coordinate system such that $\partial\Omega \cap U_z$ can be written as the graph of a Lipschitz function $g: \R^{d-1} \to \R$ in such a way that
\begin{displaymath}
	\Omega \cap U_z = \{(x_1,\ldots,x_d) : x_d < g(x_1,\ldots,x_d)\} \cap U_z
\end{displaymath}
(where we write $(x_1,\ldots,x_d) \in \R^d$). The Lipschitz normal vector to $\partial\Omega$, which {\it a priori} is a function $\nu=(\nu_1,\dots,\nu_d) \in L^\infty (\partial\Omega, \R^d)$, when restricted to $U_z$ satisfies
\begin{displaymath}
	C_z:=\essinf \{\nu_d (x) : x \in \partial\Omega \cap U_z \} > 0
\end{displaymath}
(cf.\ the proof of \cite[Lemma~6.4]{bkl19}). Together with $U_z$, we also fix a corresponding test function $\varphi_z \in C_c^\infty (\R^d)$ such that $0 \leq \varphi_z(x) \leq 1$ for all $x \in \R^d$, $\varphi_z|_{\partial\Omega \cap U_z} = 1$, and $\varphi_z(x)=0$ for all $x \in \partial\Omega$ for which $\nu_d(x) \leq 0$ (the existence of such a function $\varphi_z$ is guaranteed by shrinking $U_z$ slightly if necessary). Now choose $u \in H^2(\Omega)$ and suppose that $\|u\|_2=1$. Then
\begin{displaymath}
	 \int_{\partial\Omega\cap U_z} \left|\frac{\partial u}{\partial\nu}\right|^2\,d\mathcal{H}^{d-1}(x)
	\leq C_z^{-1}\int_{\partial\Omega} \varphi_z |\nabla u|^2 \nu_d\,d\mathcal{H}^{d-1}(x)
	= C_z^{-1}\int_\Omega \frac{\partial}{\partial x_d} (\varphi_z |\nabla u|^2)\,dx,
\end{displaymath}
where the latter equality follows from the divergence theorem applied to the vector field $$F = (0,\ldots,0,\varphi_z|\nabla u|^2) \in H^1(\Omega),$$ valid on Lipschitz domains \cite[Th\'eor\`eme~3.1.1]{necas67}. We estimate the integral on the right-hand side by
\begin{displaymath}
	\int_\Omega \frac{\partial}{\partial x_d} (\varphi_z |\nabla u|^2)\,dx \leq \|\nabla \varphi_z\|_\infty\|\nabla u\|_2^2 + 2\|\varphi_z\|_\infty
	\|D^2u\|_2 \|\nabla u\|_2.
\end{displaymath}
Since $\varphi_z$ was fixed, in dependence only on $z$, this yields \eqref{normalderivintermedest}, but for the integral over $\partial\Omega \cap U_z$ in place of $\partial\Omega$, and with $C=\max\{\|\nabla \varphi_z\|_\infty, 1\}$. Since $z \in \partial\Omega$ was arbitrary and our assumptions on $\Omega$ imply that $\partial\Omega$ is compact, a covering argument now leads to \eqref{normalderivintermedest}. As noted above, the conclusion of the lemma now follows.
\end{proof}

\begin{lemma}
Let $\Omega \subset \mathbb R^d$ be a bounded Lipschitz domain. There exist constants $C_3,C_4 > 0$ such that
\begin{displaymath}
	\int_{\partial\Omega} |u|^2\,d\mathcal{H}^{d-1}(x) \leq C_3\|D^2u\|_2^{1/2}+C_4
\end{displaymath}
for all $u \in H^2 (\Omega)$ with $\|u\|_2 = 1$.
\end{lemma}

\begin{proof}
There exists a constant $C>0$ such that
\begin{displaymath}
	\int_{\partial\Omega} |u|^2\,d\mathcal{H}^{d-1}(x) \leq C(\|\nabla u\|_2 + 1)
\end{displaymath}
for all $u \in H^1 (\Omega)$, and hence all $u \in H^2(\Omega)$, for which $\|u\|_2=1$; see \cite[Lemma~6.4]{bkl19}. Applying \eqref{nablacontrol} immediately yields the statement of the lemma.
\end{proof}


\medskip
\textbf{Estimates from above on the rate of divergence.} We are now ready to give a complementary estimate on the divergence rate of the eigenvalues \`a la \cite{dk10}. Note that Theorem~\ref{thm:divergence-basic} follows directly from Theorems~\ref{thm:num-range-eig-div} and~\ref{thm:divergence-1}.

\begin{theorem}
\label{thm:divergence-1}
Let $\Omega \subset \mathbb R^d$ be a bounded $C^1$-domain and suppose that $\sigma \in (-\frac{1}{d-1},1)$ is fixed.
\begin{itemize}
\item[(a)] If $\beta,\gamma \in \R$ are fixed, or more generally remain bounded, then for each fixed $k \in \N$,
\begin{displaymath}
	\liminf_{\alpha\to -\infty} \frac{\lambda_k (\Omega,\sigma,\alpha,\beta,\gamma)}{-|\alpha|^2} \geq \frac{1}{4}.
\end{displaymath}
\item[(b)] If $\alpha,\gamma \in \R$ are fixed, or more generally remain bounded, then for each fixed $k \in \N$,
\begin{displaymath}
	\liminf_{\beta\to -\infty} \frac{\lambda_k (\Omega,\sigma,\alpha,\beta,\gamma)}{-|\beta|^4} \geq \frac{1}{2}\left(\frac{3}{2}\right)^3.
\end{displaymath}
\item[(c)]  If $\alpha,\beta \in \R$ are fixed, or more generally remain bounded, then for each fixed $k \in \N$,
\begin{displaymath}
	\liminf_{\gamma\to -\infty} \frac{\lambda_k (\Omega,\sigma,\alpha,\beta,\gamma)}{-|\gamma|^{4/3}} \geq \left(\frac{1}{2}\right)^{1/3}\frac{3}{2}.
\end{displaymath}
\end{itemize}
Treating $\Omega$ and $\sigma$ as fixed, in each case, the estimates for $k=1$ are uniform in the other two variables, as long as these remain within a compact interval.
\end{theorem}

For $k>1$ the proof of Theorem~\ref{thm:divergence-1} makes use of a family of test functions chosen in dependence not just on the parameters themselves but on the corresponding eigenfunctions, so that it is not clear whether a uniform control on these test functions can be given; in this context we also refer back to Open Problem~\ref{prob:unif-div}.

\begin{remark}
\label{rem:big-neg-alpha}
The regime $\alpha\to-\infty$ is a partial generalisation of the results contained in \cite{abf19}. There it is proved that, for $\beta=\gamma=+\infty$, the limit inferior is actually a limit and equals precisely $1/4$. While we are unable to prove equality here, we expect it to hold; we recall the related Open Problem~\ref{prob:neg-div-more-precise} which includes, in particular, the question of the optimal constant in this bound. Note that if we have $C=1$ in \eqref{nablacontrol} then the results on the numerical range may be strengthened, and in particular the statement of Theorem~\ref{thm:num-range-eig-div}(a) made more precise, to allow us to obtain precisely $1/4$ as the correct coefficient of $-|\alpha|^2$ in the rate of divergence. While $C=1$ in \eqref{nablacontrol} is probably false in general, we remark that it would suffice if $C \sim 1$ asymptotically for the particular choice of the first eigenfunction. We also emphasise that the proof of the theorem provides the explicit bound
\begin{displaymath}
	\lambda_1 (\Omega, \sigma, \alpha, 0, 0) < -\frac{\alpha^2}{4}
\end{displaymath}
for all $\alpha < 0$ for the first Neumann eigenvalue, to be contrasted with the result of \cite[Section~4]{abf19} that the first Dirichlet (or Navier) eigenvalue is always \emph{larger} than $-\alpha^2/4$. It would be interesting to understand where the first eigenvalues of the other problems are with respect to that curve. 
\end{remark}

\begin{problem}
Study whether $\lambda_1 (\Omega, \sigma, \alpha, \beta,\gamma)$ (or, more generally, $\lambda_k (\Omega, \sigma, \alpha, \beta,\gamma)$) is larger or smaller than the quantity $-\frac{\alpha^2}{4}$, for all the possible combinations of the other parameters, namely $\Omega$, $\sigma$, $\beta$, and $\gamma$.
\end{problem}

\begin{remark}\label{rem:intermediate-regime}
Although we expect the asymptotic inequalities of Theorem \ref{thm:divergence-1} to hold also in the cases where $\beta$ and/or $\gamma=+\infty$, our proof does not cover those situations. This is due to the fact that the key exponential test function argument used in Lemma~\ref{lem:test-function-1} below does not work if we consider spaces where some boundary trace is prescribed, which is the case for instance of $H^2_0(\Omega)$ or $H^2(\Omega)\cap H^1_0(\Omega)$. Indeed, the construction of test functions for those spaces becomes extremely difficult, and we suspect that a completely different approach is needed, such as for instance the one used in \cite{abf19}.
\end{remark}

\begin{problem}
\label{op:divergence-rates}
Obtain a version of Theorem~\ref{thm:divergence-1}, or more generally Theorem~\ref{thm:divergence-basic}, when one or both of the parameters $\beta,\gamma$ are $+\infty$. This may come for free from a Dirichlet-Neumann bracketing approach (cf.\ Open Problem~\ref{prob:neg-div-more-precise}).
\end{problem}

\begin{remark}\label{rem:interplay}
Theorem \ref{thm:divergence-1} deals with the situation in which there is only one diverging parameter while the others remain controlled (recall that the cases in which $\beta$ and/or $\gamma$ go to plus infinity are actually convergent cases). If instead we allow more parameters to diverge at the same time, the behaviour of the eigenvalues become more involved, depending also on the interplay between the parameters. Such a situation of several diverging parameters can, in principle, also be studied using the same tools, but it is not clear a priori how many types of specific behaviours could be observed.
\end{remark}

\begin{problem}
Study the case(s) where more than one parameter diverges to $-\infty$, or one diverges to $-\infty$ and another to $+\infty$. Is it possible to fine-tune the divergent parameters in order to get only finitely many divergent eigenvalues?
\end{problem}

The main tool in the proof of Theorem~\ref{thm:divergence-1} is the following estimate on the form, in the special case of certain exponential test functions inspired by \cite[Lemma~2.1]{dk10} (which was in turn inspired by \cite[Theorem~2.3]{gs07}, see also \cite[Example~2.4]{lp08}).

\begin{lemma}
\label{lem:test-function-1}
Fix $\alpha,\beta,\gamma \in \R$ and $\sigma \in (-\frac{1}{d-1},1)$, let $\Omega \subset \mathbb R^d$ be a bounded $C^1$-domain, let $\xi \in \mathbb S^{d-1}$ be any unit vector, and for $t>0$ set $\varphi_t (x) := c_t e^{t\xi \cdot x}$, where $c_t>0$ is chosen so that $\|\varphi_t\|_{L^2(\Omega)}=1$. Then there exist functions $f_1, f_2 : (0,+\infty) \to \R$ such that $f_1 (t), f_2 (t) \to 1$ as $t \to \infty$, and
\begin{equation}
\label{test-function-upper-est}
	\mathcal{Q}_{\Omega,\sigma,\alpha,\beta,\gamma} (\varphi_t,\varphi_t) \leq t^4 + \alpha t^2 + 2\beta f_1(t) t^3 + 2\gamma f_2(t) t
\end{equation}
for all $t>0$.
\end{lemma}

We will see in the course of the proof that we may choose $f_1 \equiv 1$ if $\beta \geq 0$, and $f_2 \equiv 1$ if $\gamma \leq 0$.

\begin{proof}
Let $\tilde \varphi_t (x) := e^{t\xi \cdot x}$ be the corresponding ``non-normalised'' function. A short calculation gives $|D^2 \varphi_t (x)|^2 = |\Delta \varphi_t (x)|^2 = t^4 e^{2t\xi\cdot x}$, whence
\begin{equation}
\label{test-function-calc-1}
	\int_\Omega (1-\sigma) |D^2 \tilde\varphi_t |^2 + \sigma |\Delta \tilde\varphi_t |^2\,dx = t^4 \int_\Omega e^{2t\xi\cdot x}\,dx
	= t^4 \int_\Omega |\tilde\varphi_t|^2\,dx,
\end{equation}
while
\begin{equation}
\label{test-function-calc-2}
	\int_\Omega |\nabla \tilde\varphi_t |^2\,dx = t^2\int_\Omega e^{2t\xi \cdot x}\,dx = t^2 \int_\Omega |\tilde\varphi_t |^2\,dx.
\end{equation}
For the boundary integrals, if we combine the definitions
\begin{displaymath}
	f_1(t) := \frac{\int_{\partial\Omega} |\frac{\partial \tilde\varphi_t}{\partial\nu}|^2\,d\mathcal{H}^{d-1}(x)}
	{2t^3\int_\Omega |\tilde\varphi_t |^2\,dx} > 0, \qquad f_2(t) := 
	\frac{\int_{\partial\Omega} |\tilde\varphi_t|^2\,d\mathcal{H}^{d-1}(x)}{2t\int_\Omega |\tilde\varphi_t |^2\,dx} > 0,
\end{displaymath}
with \eqref{test-function-calc-1} and \eqref{test-function-calc-2}, then we clearly obtain \eqref{test-function-upper-est}.

It remains to prove that $f_1(t), f_2(t) \to 1$ as $t \to \infty$. For this, we choose the coordinate system in such a way that $\xi = (0,\ldots,0,1)$, which means that $\tilde\varphi_t(x)=e^{tx_d}$ and
\begin{equation}
\label{test-function-calc-3}
	f_1(t) = \frac{t^2\int_{\partial\Omega} e^{2tx_d}\nu_d^2\,d\mathcal{H}^{d-1}(x)}{2t^3\int_\Omega e^{2tx_d}\,dx},
	\qquad f_2(t) = \frac{\int_{\partial\Omega} e^{2tx_d}\,d\mathcal{H}^{d-1}(x)}{2t\int_\Omega e^{2tx_d}\,dx}.
\end{equation}
Now we observe that if we apply the Divergence Theorem on $\Omega$ to the vector field $F:= (0,\ldots,0,e^{2tx_d}) \in C^\infty (\R^d,\R^d)$ we obtain
\begin{displaymath}
	\int_{\partial\Omega} e^{2tx_d}\nu_d\,d\mathcal{H}^{d-1}(x) = \int_{\partial\Omega} F\cdot \nu_d\,d\mathcal{H}^{d-1}(x)
	= \int_\Omega \divergence F\,dx = 2t\int_\Omega e^{2tx_d}\,dx.
\end{displaymath}
Hence, combining this with \eqref{test-function-calc-3}, to show that $f_1(t),f_2(t) \to 1$ and hence complete the proof of the lemma, it suffices to show that the ratios
\begin{displaymath}
	\frac{\int_{\partial\Omega} e^{2tx_d}\nu_d^2\,d\mathcal{H}^{d-1}(x)}{\int_{\partial\Omega} e^{2tx_d}\nu_d\,d\mathcal{H}^{d-1}(x)} \qquad
	\text{and} \qquad \frac{\int_{\partial\Omega} e^{2tx_d}\,d\mathcal{H}^{d-1}(x)}{\int_{\partial\Omega} e^{2tx_d}\nu_d\,d\mathcal{H}^{d-1}(x)}
\end{displaymath}
both converge to $1$ as $t \to \infty$. We will only prove this for the second ratio; the first is entirely analogous. The idea is that the mass of $\tilde\varphi_t$ concentrates exponentially in the part of $\partial\Omega$ with the largest $x_d$-values; here, since $\Omega$ is $C^1$ and hence $\nu$ is continuous, $\nu_d$ is uniformly close to $1$ near these points. Formally, we argue as in \cite[Lemma~2.5]{dk10} (see also \cite[Lemma~2.4]{dk10} for properties of the level sets of $\tilde\varphi_t$).

Fix $\varepsilon > 0$ and suppose that $\tilde z = (\tilde z_1, \ldots, \tilde z_d) \in \partial\Omega$ has the largest $x_d$-coordinate of any point in $\partial\Omega$, that is, $\tilde z_d = \{\max x_d : x = (x_1,\ldots,x_d) \in \partial\Omega\}$. Denote by
\begin{displaymath}
	K_s := \{ x = (x_1,\ldots,x_d) \in \partial\Omega : \tilde z_d - x_d \leq s \},
\end{displaymath}
$s \in [0,\infty)$; then the $K_s$ are exactly the (closed) upper level sets of $\tilde \varphi_t$ in $\partial\Omega$. Since $\Omega$ is of class $C^1$, $\nu_d : \partial\Omega \to [-1,1]$ is continuous and equal to one on $K_0$. Fix $\delta>0$, to be specified precisely later, such that $\nu_d \geq 1-\delta$ on $K_\delta$. By the same argument as in \cite[Lemma~2.5]{dk10} there exists $t_0 > 0$ such that
\begin{displaymath}
	\int_{\partial\Omega \setminus K_\delta} e^{2tx_d}\,d\mathcal{H}^{d-1}(x) < \delta\int_{\partial\Omega} e^{2tx_d}\,d\mathcal{H}^{d-1}(x)
\end{displaymath}
for all $t \geq t_0$. Using that $\nu_d \geq -1$ at every point in $\partial\Omega \setminus K_\delta$ (due to the normalisation $|\nu| = 1$), it follows that
\begin{equation*}
\begin{split}
	\int_{\partial\Omega} e^{2tx_d}\nu_d\,d\mathcal{H}^{d-1}(x)
	&=\int_{K_\delta} e^{2tx_d}\nu_d\,d\mathcal{H}^{d-1}(x) + \int_{\partial\Omega\setminus K_\delta} e^{2tx_d}\nu_d\,d\mathcal{H}^{d-1}(x)\\
	&\geq (1-\delta)\int_{K_\delta} e^{2tx_d}\,d\mathcal{H}^{d-1}(x) - \int_{\partial\Omega\setminus K_\delta} e^{2tx_d}\,d\mathcal{H}^{d-1}(x)\\
	&>(1-\delta)^2\int_{\partial\Omega}e^{2tx_d}\,d\mathcal{H}^{d-1}(x) - 
		\delta\int_{\partial\Omega}e^{2tx_d}\,d\mathcal{H}^{d-1}(x)
\end{split}
\end{equation*}
for all $t \geq t_0$. Choosing $\delta$ such that $(1-\delta)^2-\delta > 1/(1+\varepsilon)$ yields
\begin{displaymath}
	1 \leq \frac{\int_{\partial\Omega} e^{2tx_d}\,d\mathcal{H}^{d-1}(x)}{\int_{\partial\Omega} e^{2tx_d}\nu_d\,d\mathcal{H}^{d-1}(x)} < 1+\varepsilon
\end{displaymath}
for all $t \geq t_0$, whence the claim.
\end{proof}

\begin{proof}[Proof of Theorem~\ref{thm:divergence-1}]
We start by showing that in each case (a), (b), (c) it is possible to choose $t$ in such a way that the Rayleigh quotient of the function $\varphi_t$ from Lemma~\ref{lem:test-function-1} yields the respective estimates. This will imply Theorem~\ref{thm:divergence-1} for $\lambda_1$. For general $k$, we will choose $k$ different vectors $\xi_1,\ldots,\xi_k$ and repeat the argument for the corresponding functions $\varphi_{\xi_1,t},\ldots,\varphi_{\xi_k,t}$.

For (a), we choose $t = (|\alpha|/2)^{1/2}$ (this corresponds to the unique solution of $\argmin_{t \ge 0} t^4 + \alpha t^2$) in \eqref{test-function-upper-est} to obtain
\begin{equation*}
	\frac{\lambda_1}{|\alpha|^2} \leq \frac{\mathcal{Q}_{\Omega,\sigma,\alpha,\beta,\gamma} (\varphi_t,\varphi_t)}{|\alpha|^{2}} 
	\leq -\frac{1}{4} + 2\beta f_1 \left(\frac{|\alpha|^{1/2}}{2^{1/2}}\right) \frac{1}{2^{3/2}} \frac{1}{|\alpha|^{1/2}}
	+2\gamma f_2 \left(\frac{|\alpha|^{1/2}}{2^{1/2}}\right) \frac{1}{2^{1/2}} \frac{1}{|\alpha|^{3/2}}
\end{equation*}
for all $\alpha < 0$ (note that this reduces to the explicit estimate $\lambda_1 < -|\alpha|^2/4$ if $\beta=\gamma=0$, where the strictness of the inequality follows from the fact that $\varphi_t$ cannot be an eigenfunction as it does not satisfy the boundary conditions). Since the terms involving $\beta$ and $\gamma$ tend to zero uniformly as long as $\beta$ and $\gamma$ remain within a compact range of values, passing to the limit as $\alpha \to -\infty$ yields (a) when $k=1$. For (b), the argument is similar, but we choose $t = 3|\beta|/2$, corresponding to $\argmin_{t \ge 0} t^4 + 2\beta t^3$, to obtain 
\begin{equation*}
	\frac{\lambda_1}{|\beta|^4} \leq \frac{\mathcal{Q}_{\Omega,\sigma,\alpha,\beta,\gamma} (\varphi_t,\varphi_t)}{|\beta|^{4}}
	\leq \left(\frac{3}{2}\right)^4 + \alpha \left(\frac{3}{2}\right)^2\frac{1}{|\beta|^2}
	- 2 f_1 \left(\frac{3|\beta|}{2}\right) \left(\frac{3}{2}\right)^3 + 2\gamma f_2 \left(\frac{3|\beta|}{2}\right) \frac{3}{2}\,\frac{1}{|\beta|^3}.
\end{equation*}
Since the terms involving $\alpha$ and $\gamma$ tend to zero uniformly for bounded $\alpha$ and $\gamma$ and $f_1 \to 1$ as $\beta \to -\infty$, passing to the limit yields (b) when $k=1$. For (c), we choose $t = (|\gamma|/2)^{1/3}$ (corresponding to $\argmin_{t \ge 0} t^4 + 2\gamma t$), which leads to
\begin{multline}
\label{gamma-test-upper}
	\frac{\lambda_1}{|\gamma|^{4/3}} \leq \frac{\mathcal{Q}_{\Omega,\sigma,\alpha,\beta,\gamma} (\varphi_t,\varphi_t)}{|\gamma|^{4/3}}
	 \leq \left(\frac{1}{2^{1/3}}\right)^4 + \alpha \left(\frac{1}{2^{1/3}}\right)^2 \frac{1}{|\gamma|^{2/3}} \\
	+ 2\beta f_1\left(\frac{|\gamma|^{1/3}}{2^{1/3}}\right) \left(\frac{1}{2^{1/3}}\right)^3 \frac{1}{|\gamma|^{1/3}}
	- 2 f_2\left(\frac{|\gamma|^{1/3}}{2^{1/3}}\right) \frac{1}{2^{1/3}}.
\end{multline}
Since the terms involving $\alpha$ and $\beta$ tend to zero uniformly for bounded $\alpha$ and $\beta$ and $f_2 \to 1$ as $\gamma \to -\infty$, passing to the limit yields (c) when $k=1$.

We now consider the case $k \geq 2$ and argue as in \cite{dk10}. The argument is by induction on $k$; we assume that the theorem is true for $k-1$ and consider $k$. We will also restrict ourselves to the case $\gamma \to -\infty$ for bounded $\alpha,\beta$; the other two cases are completely analogous. It suffices to prove that for any sequence $\gamma_n \to -\infty$ there exists a subsequence for which (c) holds for $\lambda_k$ and this subsequence. We fix such a sequence $\gamma_n$, fix some small $\delta > 0$, and denote by $u_1(\gamma_n), \ldots, u_{k-1}(\gamma_n)$ any corresponding first $k-1$ eigenfunctions, all normalised to have unit $L^2(\Omega)$-norm.

By \cite[Lemma~2.6]{dk10} there exists a direction $\xi \in \mathbb{S}^{d-1}$ such that, up to a subsequence in $n$, the corresponding function $\varphi_n := \varphi_{t_n} = ce^{t_n \xi \cdot x}$ of Lemma~\ref{lem:test-function-1} with $t_n := (|\gamma_n|/2)^{1/3}$ satisfies
\begin{equation}
\label{almost-orthogonal-test-function}
	\sum_{i=1}^{k-1} \left[\int_\Omega u_i \varphi_{n} \,dx\right]^2 \leq \delta
\end{equation}
for all $n$. The existence of such a direction $\xi$ is a rather technical result; intuitively, it follows since any such function $\varphi$ concentrates exponentially near a small part of the boundary (in dependence on $\xi$), while the $k-1$ eigenfunctions $u_1,\ldots,u_{k-1}$ cannot be large everywhere near the boundary. Observe that the function
\begin{displaymath}
	\varphi_{n} - \sum_{i=1}^{k-1} \left(\int_\Omega u_i \varphi_{n} \,dx\right) u_i ,
\end{displaymath}
is orthogonal to $u_1,\ldots,u_{k-1}$ in $L^2(\Omega)$ by construction, so we can take it as a test function for $\lambda_k$ and it follows that (cf.\ \cite[Lemma~2.3]{dk10})
\begin{displaymath}
	\lambda_k \leq \frac{\mathcal{Q}_{\Omega,\sigma,\alpha,\beta,\gamma_n}(\varphi_{n},\varphi_{n} )
	- \sum_{i=1}^{k-1}\lambda_i\left[\int_\Omega u_i \varphi_{n} \,dx\right]^2}{1-\sum_{i=1}^{k-1} \left[\int_\Omega u_i \varphi_{n} \,dx\right]^2}
\end{displaymath}
for all $n$. Using the estimate \eqref{almost-orthogonal-test-function} in the denominator and then normalising by $|\gamma_n|^{4/3}$,
\begin{displaymath}
	\frac{\lambda_1}{|\gamma_n|^{4/3}} \leq \frac{\lambda_k}{|\gamma_n|^{4/3}} 
	\leq \frac{1}{1-\delta}\left[\frac{\mathcal{Q}_{\Omega,\sigma,\alpha,\beta,\gamma_n}
	(\varphi_{n},\varphi_{n})}{|\gamma_n|^{4/3}}-\sum_{i=1}^{k-1}
	\frac{\lambda_i}{|\gamma_n|^{4/3}}\left[\int_\Omega u_i\varphi_{n}\,dx\right]^2\right].
\end{displaymath}
By the induction assumption and Theorem~\ref{thm:num-range-eig-div}, $\lambda_i/|\gamma|^{4/3}$ remains bounded as $\gamma \to -\infty$. Hence there exists a constant $C_k>0$ depending only on $k$ such that
\begin{displaymath}
	\left|\sum_{i=1}^{k-1}\frac{\lambda_i}{|\gamma_n|^{4/3}}\left[\int_\Omega u_i\varphi_{n}\,dx\right]^2\right|
	\leq C_k \sum_{i=1}^{k-1}\left[\int_\Omega u_i\varphi_{n}\,dx\right]^2 \leq C_k \delta
\end{displaymath}
for all $n$. Putting this all together,
\begin{displaymath}
	\frac{\lambda_1}{|\gamma_n|^{4/3}} \leq \frac{\lambda_k}{|\gamma_n|^{4/3}} 
	\leq \frac{1}{1-\delta}\frac{\mathcal{Q}_{\Omega,\sigma,\alpha,\beta,\gamma_n}
	(\varphi_{n},\varphi_{n})}{|\gamma_n|^{4/3}} + \frac{C_k \delta}{1-\delta}.
\end{displaymath}
Finally, we use the upper estimate \eqref{gamma-test-upper} on
\begin{displaymath}
	\frac{\mathcal{Q}_{\Omega,\sigma,\alpha,\beta,\gamma}(\varphi_{t},\varphi_{t})}{|\gamma|^{4/3}}
\end{displaymath}
obtained above for $\gamma=\gamma_n$ and $t=t_n=(|\gamma_n|/2)^{1/3}$. This leads to
\begin{displaymath}
	\liminf_{\gamma \to -\infty} \frac{\lambda_k}{-|\gamma|^{4/3}} \geq \frac{1}{1-\delta} \left(\frac{1}{2}\right)^{1/3}\frac{3}{2} 
	- \frac{C_k \delta}{1-\delta}.
\end{displaymath}
Passing to the limit as $\delta \to 0$ yields the desired estimate on $\lambda_k$. 
\end{proof}


\section{Hadamard-type formulae and criticality results}
\label{sec:shape-derivatives}

In this section we turn to the study of the dependence of the eigenvalues of problem \eqref{robinstrong} on the domain. As already mentioned in the Introduction, we know that multiple eigenvalues tend to show bifurcations, so that, in order to avoid any problem with non-differentiability, we pass to the use of elementary symmetric functions of the eigenvalues, since it has the advantage of bypassing splitting phenomena (cf.\ \cite{lala2007,lala2004}). For the sake of exposition, we will focus here just on the Full Robin problem \eqref{robin1strong}, but everything can be replicated also for the Navier--Robin problem \eqref{navier-robin} and for the Kuttler--Sigillito problem \eqref{kuttler-sigillito}. In particular, the strategy of this section will be the same introduced in \cite{lala2007,lala2004} and later applied to other Bilaplacian problems (see \cite{buoso16} and the references therein), that is we consider problem (\ref{robinstrong}) on a family of open sets parametrised by suitable diffeomorphisms $\phi$ defined on a fixed bounded open set $\Omega $ in ${\mathbb{R}}^d$, and then compute the Fr\'echet differential of elementary symmetric functions of the eigenvalues with respect to $\phi$. Note that the underlying operator (the Bilaplacian) is the same as in  \cite{buoso16}; hence, as is natural to expect, the formulae we obtain here will bear a resemblance to those presented in  \cite{buoso16}. Moreover, a number of foundational calculations we will need here were already carried out in  \cite{buoso16}; we will cite these directly rather than reproducing them in detail.

\medskip
\textbf{An Hadamard-type formula.} Let us fix a bounded open set $\Omega $ in ${\mathbb{R}}^d$ with Lipschitz boundary, and consider the following family of diffeomorphisms 
$$
{\mathcal{A}}_{\Omega }=\biggl\{\phi\in C^2(\overline\Omega\, ; {\mathbb{R}}^d ):\ \inf_{\substack{x_1,x_2\in \overline\Omega \\ x_1\ne x_2}}\frac{|\phi(x_1)-\phi(x_2)|}{|x_1-x_2|}>0 \biggr\},
$$
where $C^2(\overline\Omega\, ; {\mathbb{R}}^d )$ denotes the space of all functions from $\overline\Omega $ to ${\mathbb{R}}^d$ of class $C^2$.  Note that if $\phi \in {\mathcal{A}}_{\Omega }$ then $\phi $ is injective, Lipschitz continuous and $\inf_{\overline\Omega }|{\rm det }\nabla \phi |>0$. Moreover, $\phi (\Omega )$ is a bounded open set with Lipschitz bondary and the inverse map $\phi^{(-1)}$ belongs to  ${\mathcal{A}}_{\phi(\Omega )}$.  Thus it is possible to consider problem (\ref{robinstrong}) on $\phi (\Omega )$ and study  the dependence of $\lambda_k(\phi (\Omega ))$ on $\phi \in {\mathcal{A}}_{\Omega }$. To do so, we endow the space $C^2(\overline\Omega\, ; {\mathbb{R}}^d )$ with its usual norm.
Note that ${\mathcal{A}}_{\Omega }$ is an open set in  $C^2(\overline\Omega\, ;{\mathbb{R}}^d )$, see \cite[Lemma~3.11]{lala2004}.
    Hence, we can now study differentiability and analyticity properties of the maps $\phi \mapsto \lambda_k(\phi (\Omega ))$ defined for $\phi \in {\mathcal{A}}_{\Omega }$.
   For simplicity, we write  $\lambda_k(\phi )$ instead of $\lambda_k(\phi (\Omega ))$.
   We fix a finite set of indices $F\subset \mathbb{N}$
and we consider those maps $\phi\in {\mathcal{A}}_{\Omega }$ for which the eigenvalues
with indices in $F$  do not coincide with eigenvalues with indices not
in $F$ (i.e., we split the spectrum into two different sets that never intersect),
$$
{\mathcal { A}}_{F, \Omega }= \left\{\phi \in {\mathcal { A}}_{\Omega }:\
\lambda_k(\phi )\ne \lambda_l(\phi),\ \forall\  k\in F,\,   l\in \mathbb{N}\setminus F
\right\}.
$$
It is also convenient to consider those maps $\phi \in {\mathcal { A}}_{F, \Omega } $ such that all the eigenvalues with index in $F$
coincide and set
$$
\Theta_{F, \Omega } = \left\{\phi\in {\mathcal { A}}_{F, \Omega }:\ \lambda_{k_1}(\phi)=\lambda_{k_2}(\phi),\, \
\forall\ k_1,k_2\in F  \right\} .
$$

 For $\phi \in {\mathcal { A}}_{F, \Omega }$, the elementary symmetric functions of the eigenvalues with index in $F$ are defined by
\begin{equation*}
\Lambda_{F,s}(\phi)=\sum_{ \substack{ k_1,\dots ,k_s\in F\\ k_1<\dots <k_s} }
\lambda_{k_1}(\phi )\cdots \lambda_{k_s}(\phi ),\ \ \ s=1,\dots , |F|.
\end{equation*}
It is worth noting that, when $|F|=1$, i.e., for $F=\{k\}$, then $\Lambda_{F,s}(\phi)=\lambda_k(\phi)$.

In addition, we consider the bilinear form defined in \eqref{qf} as an operator defined on $H^2(\Omega)$ with values in its dual:
$$
P[u][v]=Q(u,v).
$$
Here we separate the two arguments $u,v\in H^2(\Omega)$ to highlight that $P$ takes the argument $u\in H^2(\Omega)$ to the image $P[u]$ which is a functional in the dual of $H^2(\Omega)$ acting like $Q(u,\cdot)$. In the same manner, we define the operator (mapping $H^2(\Omega)$ to its dual)
$$
J[u][v]=\int_\Omega uv \,dx,
$$
so that the eigenvalue problem \eqref{robinweak} can be rewritten in a distributional sense as
\begin{equation*}
P[u]=\lambda J[u].
\end{equation*}

We have the following 

\begin{theorem}
\label{duesettesys}
Let $\Omega $ be a Lipschitz bounded open set in ${\mathbb{R}}^d$ and  $F$ be a finite set in  ${\mathbb{N}}$. 
The set ${\mathcal { A}}_{F, \Omega }$ is open in
$\mathcal{A}_{\Omega}$, and the real-valued maps 
$ \Lambda_{F,s}$ are real-analytic on  ${\mathcal { A}}_{F, \Omega }$, for all $s=1,\dots , |F|$. 
Moreover, if $\tilde \phi\in \Theta_{F, \Omega }  $ is such that the eigenvalues $\lambda_k(\tilde \phi)$ assume the common value $\lambda_F(\tilde \phi )$ for all $k\in F$, and  $\tilde\Omega=\tilde \phi (\Omega )$ is of class $C^{4}$ then  the Fr\'{e}chet differential of the map $\Lambda_{F,s}$ at the point $\tilde\phi $ is given by the formula
\begin{equation}
\label{derivdsys}
d|_{\phi=\tilde{\phi}}(\Lambda_{F,s})[\psi]
=	\lambda_F^{s-1}(\tilde{\phi})\binom{|F|-1}{s-1} \sum_{l=1}^{|F|}\int_{\partial\tilde\Omega} G(v_l)\zeta\cdot\nu \,d\mathcal{H}^{d-1}(x),
			\end{equation}
for all $\psi\in C^2(\overline\Omega;\mathbb{R}^d)$, where $\zeta=\psi\circ\tilde{\phi}^{-1}$, $K$ is the mean curvature of $\partial\tilde{\phi}(\Omega)$ (i.e.\ $K= \divergence \nu$), $\{v_l\}_{l\in F}$ is an orthonormal basis in $H^2(\tilde \phi (\Omega ))$ of the eigenspace associated with $\lambda_F(\tilde \phi)$ (the orthonormality being taken with respect to the $L^2$-inner product), and $G(v)$ is given by the formula 
\begin{multline}
\label{Gvl} 
G(v)=(1-\sigma)|D^2 v|^2+\sigma(\Delta v)^2+\alpha|\nabla v|^2\\
			+2\beta \frac{\partial v}{\partial\nu}\Delta v +2\beta \nabla \frac{\partial v}{\partial\nu}\nabla v
			-\beta K  \left(\frac{\partial v}{\partial\nu}\right)^2-\beta\frac{\partial v}{\partial\nu}\frac{\partial^2 v}{\partial\nu^2}
			+\gamma K v^2+2 \gamma v\frac{\partial v}{\partial\nu}.
\end{multline}
\end{theorem}

\begin{proof}
For the proof of the first part of the theorem we refer to \cite[Theorem 3.1]{bula2013} (see also \cite{lala2004}). Concerning formula \eqref{derivdsys},
we start by recalling that, for $\phi\in\mathcal{A}_{\Omega}$, we have that the pull-back of the operator $P$ is defined by
$$
P_{\phi}[u][v]=\mathcal{Q}_{\phi(\Omega)}(u\circ\phi^{-1},v\circ\phi^{-1}),
$$
for any $u,v\in H^2(\Omega)$, and similarly 
$$
J_{\phi}[u][v]=\int_\Omega uv |\det \nabla \phi|\,dx.
$$
We have
\begin{equation*}
d|_{\phi =\tilde\phi}\Lambda_{F,s}[\psi]
=-\lambda_{F}^{s-1}(\tilde \phi)   \binom{|F|-1}{s-1}
\sum_{l\in F}
\left(\lambda_F(\tilde\phi)(d|_{\phi=\tilde{\phi}}J_{\phi}[\psi])[u_l][u_l]
 -(d|_{\phi=\tilde{\phi}}P_{\phi}[\psi])[u_l][u_l]\right),
\end{equation*}
for all  $\psi\in C^2( \overline\Omega\, ;{\mathbb{R}}^d) $ (cf.\ \cite[proof of Theorem 3]{buoso16}), where $ u_l= v_l\circ \tilde \phi $ for all $l\in F$. Note also that by standard regularity theory (see e.g., \cite[Theorem~2.20]{ggs}) $v_l\in H^{4}(\tilde \phi (\Omega ))$ for all $l\in F$.

Formula \eqref{derivdsys} is then obtained using \cite[Lemmas 7, 8]{buoso16} and Lemma \ref{conticino} as follows. We first observe that, in the notation of \cite{buoso16},
$$
P_\phi=(1-\sigma)M_\varphi+\sigma B_\varphi+\alpha L_\varphi+\beta T_\varphi  +\gamma J_{3,\varphi},
$$
where $T_\varphi$ is defined in Lemma \ref{conticino}. Formula \eqref{derivdsys} now follows just by adding together all the terms in the differential and noticing that a number of summands vanish due to the boundary conditions.
\end{proof}

We note that a Hadamard-type formula like \eqref{derivdsys} has been already derived for other type of boundary conditions (Dirichlet, Neumann, etc.), we refer to \cite{buoso16} for an overview.

\begin{lemma}
\label{conticino}
Let $T$ be the operator from $H^2(\Omega)$ to its dual defined by
$$
T[u_1][u_2]=\int_{\partial\Omega}\frac{\partial u_1}{\partial\nu}\frac{\partial u_2}{\partial\nu} \,d\mathcal{H}^{d-1}(x),
$$
and, for $\phi\in \mathcal{A}_{\Omega}$, let $T_{\phi}$ be the pull-back
$$
T_{\phi}[u_1][u_2]=\int_{\partial\phi(\Omega)}\frac{\partial (u_1\circ\phi^{-1})}{\partial\nu}\frac{\partial (u_2\circ\phi^{-1})}{\partial\nu} \,d\mathcal{H}^{d-1}(x).
$$
Also let $u_1,u_2\in H^2(\Omega)$ be such that $v_1=u_1\circ\tilde{\phi}^{-1}$, $v_2=u_2\circ\tilde{\phi}^{-1}\in H^{4}(\tilde{\phi}(\Omega))$. Then
\begin{multline*}
(d|_{\phi=\tilde{\phi}}T_{\phi}[\psi])[u_1][u_2]
=\int_{\partial\tilde\Omega}\left(\frac{\partial v_1}{\partial\nu}\Delta v_2+\frac{\partial v_2}{\partial\nu}\Delta v_1\right)\zeta\cdot\nu \,d\mathcal H^{d-1}(x)\\
+\int_{\partial\tilde\Omega}\left(\nabla\frac{\partial v_1}{\partial\nu}\nabla v_2+\nabla\frac{\partial v_2}{\partial\nu}\nabla v_1\right)\zeta\cdot\nu \,d\mathcal H^{d-1}(x)
-\int_{\partial\tilde\Omega}\frac{\partial\ }{\partial\nu}\left(\frac{\partial v_1}{\partial\nu}\frac{\partial v_2}{\partial\nu}\right)\zeta\cdot\nu \,d\mathcal H^{d-1}(x)\\
-\int_{\partial\tilde\Omega}K\frac{\partial v_1}{\partial\nu}\frac{\partial v_2}{\partial\nu}\zeta\cdot\nu \,d\mathcal H^{d-1}(x)
-\int_{\partial\tilde\Omega}\left(\frac{\partial v_1}{\partial\nu}\nabla v_2+\frac{\partial v_2}{\partial\nu}\nabla v_1\right)\frac{\partial \zeta}{\partial\nu} \,d\mathcal H^{d-1}(x)\\
-\int_{\partial\tilde\Omega}\left(\frac{\partial v_1}{\partial\nu}\frac{\partial\ }{\partial\nu}\nabla v_2+\frac{\partial v_2}{\partial\nu}\frac{\partial\ }{\partial\nu}\nabla v_1\right)\zeta \,d\mathcal H^{d-1}(x),
	\end{multline*}
for all $\psi\in C^2(\overline{\Omega};\mathbb{R}^d)$, where $\tilde\Omega=\tilde\phi(\Omega)$, $\zeta=\psi\circ\tilde{\phi}^{-1}$, and $K$ is the mean curvature of $\partial\tilde{\phi}(\Omega)$, i.e., $K=\divergence \nu$. 
\end{lemma}

\begin{proof}
First of all, we recall that the normals of $\Omega$ and of $\phi(\Omega)$ are related by the formula
$$
\nu_{\phi(\Omega)}\circ\phi=\frac{\nabla\phi^{-T}\nu_{\Omega}}{\left|\nabla\phi^{-T}\nu_{\Omega}\right|},
$$
where $\nabla\phi^{-T}=(\nabla\phi^{-1})^T$. Using this with \cite[Formula (4.11)]{delzol} we may rewrite $T_{\phi}$ as
\begin{multline*}
T_{\phi}[u_1][u_2]=\int_{\partial\Omega} \left[\nabla u_1 (\nabla \phi)^{-1}\frac{(\nabla\phi)^{-T}\nu_{\Omega}}{\left|(\nabla\phi)^{-T}\nu_{\Omega}\right|}\right]\times \\
\times \left[\nabla u_2 (\nabla \phi)^{-1}\frac{(\nabla\phi)^{-T}\nu_{\Omega}}{\left|(\nabla\phi)^{-T}\nu_{\Omega}\right|}\right]\left|(\nabla\phi)^{-T}\nu_{\Omega}\right||\det \nabla\phi|\,d\mathcal H^{d-1}(x),
\end{multline*}
and therefore
\begin{multline*}
(d|_{\phi=\tilde{\phi}}T_{\phi}[\psi])[u_1][u_2]
=-\int_{\partial\tilde\Omega}\left(\frac{\partial v_1}{\partial\nu}\nabla v_2+\frac{\partial v_2}{\partial\nu}\nabla v_1\right)\cdot\left(\nabla\zeta+\nabla\zeta^T\right)\nu \,d\mathcal H^{d-1}(x)\\
+\frac12\int_{\partial\tilde\Omega}\frac{\partial v_1}{\partial\nu}\frac{\partial v_2}{\partial\nu}\nu^T\left(\nabla\zeta+\nabla\zeta^T\right)\nu \,d\mathcal H^{d-1}(x)
+\int_{\partial\tilde\Omega}\frac{\partial v_1}{\partial\nu}\frac{\partial v_2}{\partial\nu}\divergence \zeta \,d\mathcal H^{d-1}(x).
\end{multline*}

Now, keeping in mind that $\nu$ denotes the normal unit vector to $\tilde\Omega$ and $F_{\partial\tilde\Omega}=F-(F\cdot\nu)\nu$ is the tangential part of the vector $F$, we see that
\begin{multline*}
-\int_{\partial\tilde\Omega}\left(\frac{\partial v_1}{\partial\nu}\nabla v_2+\frac{\partial v_2}{\partial\nu}\nabla v_1\right)\cdot\left(\nabla\zeta+\nabla\zeta^T\right)\nu \,d\mathcal H^{d-1}(x)\\
=-4\int_{\partial\tilde\Omega}\frac{\partial v_1}{\partial\nu}\frac{\partial v_2}{\partial\nu}\frac{\partial \zeta}{\partial\nu}\nu \,d\mathcal H^{d-1}(x)\\
-\int_{\partial\tilde\Omega}\left(\frac{\partial v_1}{\partial\nu}\nabla_{\partial\tilde\Omega} v_2+\frac{\partial v_2}{\partial\nu}\nabla_{\partial\tilde\Omega} v_1\right)\cdot\left(\left.\frac{\partial\zeta}{\partial\nu}\right|_{\partial\tilde\Omega}+\nabla_{\partial\tilde\Omega}\zeta^T\cdot \nu\right) \,d\mathcal H^{d-1}(x)\\
=-4\int_{\partial\tilde\Omega}\frac{\partial v_1}{\partial\nu}\frac{\partial v_2}{\partial\nu}\frac{\partial \zeta}{\partial\nu}\nu \,d\mathcal H^{d-1}(x)
-\int_{\partial\tilde\Omega}\left(\frac{\partial v_1}{\partial\nu}\nabla_{\partial\tilde\Omega} v_2+\frac{\partial v_2}{\partial\nu}\nabla_{\partial\tilde\Omega} v_1\right)\cdot\left.\frac{\partial\zeta}{\partial\nu}\right|_{\partial\tilde\Omega} \,d\mathcal H^{d-1}(x)\\
+\int_{\partial\tilde\Omega}{\rm div}_{\partial\tilde\Omega}\left(\frac{\partial v_1}{\partial\nu}\nabla_{\partial\tilde\Omega} v_2+\frac{\partial v_2}{\partial\nu}\nabla_{\partial\tilde\Omega} v_1\right)\zeta\cdot\nu \,d\mathcal H^{d-1}(x)\\
+\int_{\partial\tilde\Omega}\left(\frac{\partial v_1}{\partial\nu}\nabla_{\partial\tilde\Omega} v_2+\frac{\partial v_2}{\partial\nu}\nabla_{\partial\tilde\Omega} v_1\right)\nabla_{\partial\tilde\Omega}\nu\cdot\zeta_{\partial\tilde\Omega} \,d\mathcal H^{d-1}(x).
\end{multline*}

We also recall that ${\rm div}\, F={\rm div}_{\partial\tilde\Omega} F+\frac{\partial F}{\partial\nu}\cdot\nu$, so
\begin{multline*}
\int_{\partial\tilde\Omega}\frac{\partial v_1}{\partial\nu}\frac{\partial v_2}{\partial\nu}\divergence \zeta \,d\mathcal H^{d-1}(x)
=\int_{\partial\tilde\Omega}\frac{\partial v_1}{\partial\nu}\frac{\partial v_2}{\partial\nu}\frac{\partial \zeta}{\partial\nu}\cdot\nu\,d\mathcal H^{d-1}(x)\\
+\int_{\partial\tilde\Omega}K\frac{\partial v_1}{\partial\nu}\frac{\partial v_2}{\partial\nu} \zeta\cdot\nu \,d\mathcal H^{d-1}(x)
-\int_{\partial\tilde\Omega}\nabla_{\partial\tilde\Omega}\left(\frac{\partial v_1}{\partial\nu}\frac{\partial v_2}{\partial\nu}\right) \cdot\zeta_{\partial\tilde\Omega}\cdot\nu \,d\mathcal H^{d-1}(x).
\end{multline*}

Collecting these facts we obtain
\begin{multline*}
(d|_{\phi=\tilde{\phi}}T_{\phi}[\psi])[u_1][u_2]
=-2\int_{\partial\tilde\Omega}\frac{\partial v_1}{\partial\nu}\frac{\partial v_2}{\partial\nu}\frac{\partial \zeta}{\partial\nu}\nu \,d\mathcal H^{d-1}(x)\\
-\int_{\partial\tilde\Omega}\left(\frac{\partial v_1}{\partial\nu}\nabla_{\partial\tilde\Omega}v_2+\frac{\partial v_2}{\partial\nu}\nabla_{\partial\tilde\Omega}v_1\right)\left(\left.\frac{\partial\zeta}{\partial\nu}\right|_{\partial\tilde\Omega}-\nabla_{\partial\tilde\Omega}\nu\cdot\zeta_{\partial\tilde\Omega}\right)\,d\mathcal H^{d-1}(x)\\
+\int_{\partial\tilde\Omega}\divergence_{\partial\tilde\Omega}\left(\frac{\partial v_1}{\partial\nu}\nabla_{\partial\tilde\Omega}v_2+\frac{\partial v_2}{\partial\nu}\nabla_{\partial\tilde\Omega}v_1\right) \zeta\cdot\nu \,d \mathcal H^{d-1}(x)
+\int_{\partial\tilde\Omega} K \frac{\partial v_1}{\partial\nu}\frac{\partial v_2}{\partial\nu} \zeta\cdot\nu  \,d\mathcal H^{d-1}(x)\\
-\int_{\partial\tilde\Omega}\nabla_{\partial\tilde\Omega}\left(\frac{\partial v_1}{\partial\nu}\frac{\partial v_2}{\partial\nu}\right)\cdot\zeta_{\partial\tilde\Omega} \,d\mathcal H^{d-1}(x).
\end{multline*}

Now we observe that

\begin{multline*}
-\int_{\partial\tilde\Omega}\nabla_{\partial\tilde\Omega}\left(\frac{\partial v_1}{\partial\nu}\frac{\partial v_2}{\partial\nu}\right)\cdot\zeta_{\partial\tilde\Omega} \,d\mathcal H^{d-1}(x)\\
=-\int_{\partial\tilde\Omega}\nabla\left(\frac{\partial v_1}{\partial\nu}\frac{\partial v_2}{\partial\nu}\right)\cdot\zeta \,d\mathcal H^{d-1}(x)
+\int_{\partial\tilde\Omega} \frac{\partial \ }{\partial\nu}\left(\frac{\partial v_1}{\partial\nu}\frac{\partial v_2}{\partial\nu}\right)\zeta\cdot\nu \,d\mathcal H^{d-1}(x)\\
=-\int_{\partial\tilde\Omega}\left(\frac{\partial v_1}{\partial\nu}\nabla\frac{\partial v_2}{\partial\nu}+\frac{\partial v_2}{\partial\nu}\nabla\frac{\partial v_1}{\partial\nu}\right)\cdot\zeta \,d\mathcal H^{d-1}(x)
+\int_{\partial\tilde\Omega} \frac{\partial \ }{\partial\nu}\left(\frac{\partial v_1}{\partial\nu}\frac{\partial v_2}{\partial\nu}\right)\zeta\cdot\nu \,d\mathcal H^{d-1}(x)\\
=-\int_{\partial\tilde\Omega}\left(\frac{\partial v_1}{\partial\nu}\nabla_{\partial\tilde\Omega} v_2+\frac{\partial v_2}{\partial\nu}\nabla_{\partial\tilde\Omega} v_1\right)\cdot \nabla_{\partial\tilde\Omega}\nu\cdot\zeta \,d\mathcal H^{d-1}(x)\\
-\int_{\partial\tilde\Omega}\left(\frac{\partial v_1}{\partial\nu}\frac{\partial \ }{\partial\nu}\nabla v_2+\frac{\partial v_2}{\partial\nu}\frac{\partial \ }{\partial\nu}\nabla v_1\right)\cdot\zeta \,d\mathcal H^{d-1}(x)
+\int_{\partial\tilde\Omega} \frac{\partial \ }{\partial\nu}\left(\frac{\partial v_1}{\partial\nu}\frac{\partial v_2}{\partial\nu}\right)\zeta\cdot\nu \,d\mathcal H^{d-1}(x),
\end{multline*}
so that, after rearranging the terms, we have
\begin{multline*}
(d|_{\phi=\tilde{\phi}}T_{\phi}[\psi])[u_1][u_2]
=\int_{\partial\tilde\Omega}\frac{\partial\ }{\partial\nu}\left(\frac{\partial v_1}{\partial\nu}\frac{\partial v_2}{\partial\nu}\right)\zeta\cdot\nu \,d\mathcal H^{d-1}(x)\\
+\int_{\partial\tilde\Omega}K\frac{\partial v_1}{\partial\nu}\frac{\partial v_2}{\partial\nu}\zeta\cdot\nu \,d\mathcal H^{d-1}(x)
+\int_{\partial\tilde\Omega}\divergence_{\partial\tilde\Omega}\left(\frac{\partial v_1}{\partial\nu}\nabla_{\partial\tilde\Omega} v_2+\frac{\partial v_2}{\partial\nu}\nabla_{\partial\tilde\Omega} v_1\right)\zeta\cdot\nu \,d\mathcal H^{d-1}(x)\\
-\int_{\partial\tilde\Omega}\left(\frac{\partial v_1}{\partial\nu}\nabla v_2+\frac{\partial v_2}{\partial\nu}\nabla v_1\right)\frac{\partial \zeta}{\partial\nu} \,d\mathcal H^{d-1}(x)\\
-\int_{\partial\tilde\Omega}\left(\frac{\partial v_1}{\partial\nu}\frac{\partial\ }{\partial\nu}\nabla v_2+\frac{\partial v_2}{\partial\nu}\frac{\partial\ }{\partial\nu}\nabla v_1\right)\zeta \,d\mathcal H^{d-1}(x).
\end{multline*}

The conclusion now comes developing the tangential divergence and using the identity $\left.\Delta\right|_{\partial\tilde\Omega}=\frac{\partial^2\ }{\partial\nu^2}+K \frac{\partial\ }{\partial\nu}+\Delta_{\partial\tilde\Omega}$.

\end{proof}

	
\medskip
\textbf{Applications to shape optimisation and criticality.} As we now have the formula for the shape derivatives of the (elementary symmetric functions of the) eigenvalues, we may then attack the problem of shape optimisation using the Lagrange Multiplier Theorem. In particular, we have the following theorem as an immediate consequence, providing a characterisation of all critical domain transformations $\phi$ (see also \cite{buoso16}).


\begin{theorem}\label{moltiplicatorisys}
Let $\Omega $ be a Lipschitz bounded open set in ${\mathbb{R}}^d$,  and $F$ be a finite subset of ${\mathbb{N}}$.
Assume that $\tilde \phi\in \Theta_{F, \Omega }$ is such that $\tilde \phi (\Omega )$ is of class $C^{4}$ and that the eigenvalues $\lambda_j(\tilde \phi )$ have the common value $\lambda_F(\tilde \phi )$ for all $j\in F$.  Let $\{  v_l\}_{l\in F}$ be an orthornormal basis in $H^2(\tilde \phi (\Omega ))$ of the eigenspace corresponding to $\lambda_F(\tilde \phi)$ (the orthonormality being taken with respect to the $L^2$-inner product). Then $\tilde \phi$
is a critical point  for any of the functions $\Lambda_{F,s}$, $s=1,\dots , |F|$,  with  volume constraint
if and only if  there exists $c\in {\mathbb{R}}$ such that
	\begin{equation}
	\label{lacondizionedsys}
	\sum_{l=1}^{|F|}G(v_l)=c\ {\rm\ on\ }\partial\tilde\phi(\Omega),
	\end{equation}
where $G$ is given by \eqref{Gvl}.
Similarly, $\tilde \phi$
is a critical point  for any of the functions $\Lambda_{F,s}$, $s=1,\dots , |F|$,  with  perimeter constraint
if and only if  there exists $c\in {\mathbb{R}}$ such that
	\begin{equation}
	\label{lacondizionedsysp}
	\sum_{l=1}^{|F|}G(v_l)=cK\ {\rm\ on\ }\partial\tilde\phi(\Omega),
	\end{equation}
where $K$ is the mean curvature of $\partial\tilde\phi(\Omega)$.
\end{theorem}
\begin{proof}We recall that, if we set $\mathcal{V}(\phi)=|\phi(\Omega)|$ and $\mathcal{P}(\phi)=|\partial\phi(\Omega)|$, then we have
\begin{equation*}
d|_{\phi=\tilde{\phi}}\mathcal V (\phi)[\psi]=\int_{\partial\tilde\phi(\Omega)}(\psi\circ\tilde{\phi}^{-1})\cdot\nu \,d\mathcal H^{d-1}(x),
\end{equation*}
and
\begin{equation*}
d|_{\phi=\tilde{\phi}}\mathcal P (\phi)[\psi]=\int_{\partial\tilde\phi(\Omega)}K(\psi\circ\tilde{\phi}^{-1})\cdot\nu \,d\mathcal H^{d-1}(x),
\end{equation*}
see e.g., \cite{lambertisteklov}. The result follows combining the Lagrange Multipliers Theorem with formulae \eqref{derivdsys}, \eqref{lacondizionedsys}, and \eqref{lacondizionedsysp}.
\end{proof}

As mentioned, and as is generally well known, balls usually play a central role in shape optimisation problems for the eigenvalues of the Laplacian, but also of biharmonic operators. For rotation invariant operators such as the Bilaplacian, it is easy to show that any eigenfunction associated with a simple eigenvalue is radial. On the other hand, when dealing with a multiple eigenvalue, we cannot consider the eigenfunctions alone, but we have to consider the whole eigenspace. In particular, when this observation is coupled with condition \eqref{lacondizionedsys}, balls turn out to enjoy a nice criticality property.

	\begin{theorem}
	\label{lepallesys}
	Let $\Omega$ be a bounded Lipschitz domain in $\mathbb{R}^d$. Let $\tilde{\phi}\in\mathcal{A}_{\Omega}$ be such that
	$\tilde{\phi}(\Omega)$ is a ball. Let $\tilde{\lambda}$ be an eigenvalue of problem (\ref{robinstrong}) in $\tilde{\phi}(\Omega)$,
	and let $F$ be the set of $j\in\mathbb{N}$ such that $\lambda_j(\tilde{\phi})=\tilde{\lambda}$.
	Then $\tilde{\phi}$ is a critical point for $\Lambda_{F,s}$ both under a volume constraint and under a perimeter constraint,
	for all $s=1,\dots,|F|$.
	\end{theorem}
	
	\begin{proof} Since the mean curvature $K$ is constant on the boundary of any ball, in order to prove the theorem it suffices to show that $\sum_{l=1}^{|F|}G(v_l)$ is constant on the boundary, then the claim will follow from Theorem~\ref{moltiplicatorisys}. In particular, we show now that all its components are radial.
	
		First of all, note that by standard regularity theory (see e.g., \cite{ggs}), $v_{j}\in C^{\infty}(\overline B)$ for all $j\in F$. Now, let $O_d(\mathbb{R})$ denote the group of orthogonal linear transformations in $\mathbb{R}^d$. Since the operators $P$ and $J_i$, $i=1,2,3$ are invariant under rotations, then
		$v_k\circ R$, where $R\in O_d(\mathbb R)$, is still an eigenfunction with eigenvalue $\lambda$; moreover, $\{v_j\circ R:j=1,\dots, |F|\}$ is another orthonormal basis
		for the eigenspace associate with $\lambda$. Since both $\{v_j:j=1,\dots, |F|\}$ and $\{v_j\circ R:j=1,\dots, |F|\}$
		are orthonormal bases, then there exists $A[R]\in O_d(\mathbb{R})$ with matrix $(A_{ij}[R])_{i,j=1,\dots,|F|}$ such that
		\begin{equation}
		\label{eq}
		v_j=\sum_{l=1}^{|F|}A_{jl}[R]v_l\circ R.
		\end{equation}
		This implies that
		\begin{equation*}
		\sum_{j=1}^{|F|}v_j^2=\sum_{j=1}^{|F|}(v_j\circ R)^2,
		\end{equation*}
		from which we get that $\sum_{j=1}^{|F|}v_j^2$ is radial. In particular, one can also show that
			$$
		\sum_{j=1}^{|F|}v_j^2,\quad \sum_{j=1}^{|F|}|\nabla v_j|^2,\quad \sum_{j=1}^{|F|}|\Delta v_j|^2,\quad \sum_{j=1}^{|F|}|D^2v_j|^2
		$$ 
	are radial functions, following the lines of \cite[Theorem 5]{buoso16}.
	
	From \eqref{eq} we also obtain
	$$
	\frac{\partial v_j}{\partial r}=\sum_{l=1}^{|F|}A_{jl}[R]\frac{\partial v_l}{\partial r}\circ R
	$$
and
$$
\Delta v_j=\sum_{l=1}^{|F|}A_{jl}[R](\Delta v_l)\circ R.
$$
Combining all these formulae tells us that 
$$
\sum_{j=1}^{|F|}\frac{\partial v_j}{\partial r}\Delta v_j,\quad
  \sum_{j=1}^{|F|}\nabla \frac{\partial v_j}{\partial r}\nabla v_j,\quad 
			\sum_{j=1}^{|F|}\left(\frac{\partial v_j}{\partial r}\right)^2,\quad
			\sum_{j=1}^{|F|}\frac{\partial v_j}{\partial r}\frac{\partial^2 v_j}{\partial r^2},\quad 
			\sum_{j=1}^{|F|} v_j\frac{\partial v_j}{\partial r}
$$
are radial functions. The fact that the radial derivative coincides with the normal derivative at the boundary concludes the proof.
\end{proof}

In general, balls can be expected to be the extremisers only when the first eigenvalue is involved (see e.g., \cite{henrot}), and even in this case one can find counterexamples in certain situations \cite{fgw,frekrej15,kuttler72}. Moreover, in this case it is not clear whether the conjecture that the ball is the extremiser for all possible choices of the parameters would make sense. On the contrary, it is possible that the Robin Bilaplacian \eqref{robin1strong} shows a behaviour similar to that of the Robin Laplacian:  while the Bossel--Daners inequality asserts that the ball is the minimiser of $\mu_1^R(\gamma)$ under a volume constraint for all $\gamma > 0$ \cite{bossel,bd10}, Bareket's conjecture that the ball is a maximiser under a volume constraint for all $\gamma < 0$ \cite{bareket} is now known to be false; it is true for sufficiently small negative $\gamma$, but for large negative $\gamma$ annuli are known to have larger eigenvalues, and are conjectured to be maximisers in this case (see \cite{afk17,frekrej15}). While Theorem~\ref{lepallesys} shows that the ball is always a critical point for the first eigenvalue of the Robin Bilaplacian, it is not clear whether it is a local extremiser since this would require the computation of the second order shape derivative, which does not appear to be an easy task.

\begin{problem}
\label{prob:min-exist}
Show that the problem of minimising $\lambda_k (\Omega,\sigma,\alpha,\beta,\gamma)$ among all bounded Lipschitz domains $\Omega \subset \R^d$ of fixed volume or fixed perimeter is well posed, if the parameters $\sigma \in (-\frac{1}{d-1},1)$, $\alpha \geq 0$, $\beta \geq 0$ and $\gamma \geq 0$ are all fixed (and $(\beta,\gamma) \neq (0,0)$).
\end{problem}

\begin{problem}
Under the assumptions of Open Problem~\ref{prob:min-exist}, prove that the ball minimises $\lambda_1 (\Omega,\sigma,\alpha,\beta,\gamma)$ among all bounded Lipschitz domains $\Omega \subset \R^d$ of fixed volume or fixed perimeter.
\end{problem}

Due to the presence of multiple parameters for the Bilaplacian, unlike in the case of the Robin Laplacian the question of when to seek a minimiser and when to seek a maximiser becomes more complicated, if not all parameters have the same sign.

\begin{problem}
Study when the problem admits a minimiser, and when it admits a maximiser, in dependence on $\alpha,\beta,\gamma \in \R$.
\end{problem}

\begin{problem}
Establish a counterpart to the counterexample to Bareket's conjecture. That is, prove that for $\beta < 0$ and/or $\gamma < 0$ sufficiently large negative, the ball does not maximise $\lambda_1 (\Omega,\sigma,\alpha,\beta,\gamma)$. It should be sufficient to compare the ball with suitable annuli (cf.\ \cite{frekrej15} and also Open Problem~\ref{prob:simple-study}).
\end{problem}

Under the assumption that Open Problem~\ref{prob:min-exist} can be proved, the related question of minimising or maximising the higher eigenvalues becomes interesting. In general, as mentioned, balls cannot be expected to always be the extremisers, even if Theorem~\ref{lepallesys} shows that balls are critical domains for all the eigenvalues. In this regard, the case $\sigma=0, \alpha>0, \beta=0$ has been recently covered in \cite{chaslang} where the authors prove that, for small negative values of $\gamma$ the ball is the only maximiser under a volume constraint.

\begin{problem}
Study the minimisation/maximisation problem for $\lambda_k (\Omega,\sigma,\alpha,\beta,\gamma)$, $k\in\mathbb N$ among all (bounded Lipschitz) domains $\Omega \subset \R^d$ of fixed volume or fixed perimeter.
\end{problem}


\section*{Acknowledgements}

The authors wish to thank two anonymous referees for their detailed and thoughtful comments on an earlier, submitted version of this paper, which greatly helped to improve it. A significant part of the research in this paper was carried out while the first author held a post-doctoral position at the \'Ecole Polytechnique F\'ed\'erale de Lausanne. 
This work was partially supported by the Funda\c c\~{a}o para a Ci\^{e}ncia e a Tecnologia
(Portugal) through the program ``Investigador FCT'' with reference IF/00177/2013 (D.B.) and IF/01461/2015 (J.B.K.) and the project {\it Extremal spectral quantities and related problems}, reference PTDC/MAT-CAL/4334/2014 (both authors), as well as by SNSF project ``Bounds for the Neumann and Steklov eigenvalues of the biharmonic operator'', SNSF grant number 200021\_178735 (D.B.).
The first author wishes to express his gratitude to the University of Lisbon for its hospitality that helped the development of this paper. 
The first author is a member of the Gruppo Nazionale per l'Analisi
Matematica, la Probabilit\`a e le loro Applicazioni (GNAMPA) of the Istituto Naziona\-le di Alta Matematica (INdAM).

\bibliographystyle{plain}
\bibliography{references_buoso_kennedy}
\end{document}